\documentclass[twoside, 11pt]{article}
\usepackage{quoting}
\usepackage[letterpaper,margin=1in]{geometry}
\usepackage{enumitem}
\usepackage{amsthm}
\usepackage{tikz}
\usetikzlibrary{decorations.pathreplacing,calligraphy}

\usepackage{amsmath}

\usepackage[linktocpage=true,pagebackref=true,colorlinks,linkcolor=magenta,citecolor=blue,bookmarks,bookmarksopen,bookmarksnumbered,ocgcolorlinks]{hyperref}
\usepackage{comment}
\usepackage{lineno}
\usepackage{array}
\usepackage[nameinlink]{cleveref}
\usepackage[nottoc]{tocbibind}
\usepackage{pifont}
\usetikzlibrary{calc}
\usepackage{hyperref}
\usepackage{amsthm, amsmath, amssymb}
\usepackage{enumerate}
\usepackage{mathrsfs}
\usepackage{mathtools}
\usepackage{ragged2e}
\usepackage{algpseudocode,algorithm}
\usepackage{cleveref}
\usepackage{lineno}
\usepackage{comment}
\usepackage{todonotes}
\usepackage{amsthm}
\usepackage{tikz}
\usepackage{circuitikz}
\usetikzlibrary{decorations.pathreplacing,calligraphy}
\usepackage{caption}
\usepackage{subcaption}
\usepackage{wrapfig}
\usepackage{amsmath}

\newtheorem{fact}{Fact}

\usepackage[charter]{mathdesign}

\newcommand{\ZZ}{\mathbb{Z}}
\newcommand{\OO}{\mathcal{O}}

\newcommand{\UG}{\langle u \rangle}

\newcommand{\GA}{\Gamma_{\OO}}
\newcommand{\XO}{X_{\mathcal{O}}}
\newcommand{\XX}{\mathfrak{X}}
\newcommand{\al}{\alpha}
\newcommand{\be}{\beta}
\newcommand{\ra}{\xrightarrow{}}
\newcommand{\ang}[1]{\langle #1 \rangle}
\newcommand{\Gprime}{\mathcal{G}_{PP}}

\newtheorem{definition}{Definition}
\newtheorem{observation}{Observation}
\newtheorem{remark}{Remark}
\newtheorem{theorem}{Theorem}
\newtheorem{lemma}{Lemma}
\newtheorem{corollary}{Corollary}

\newtheorem{construction}{Construction}
\newtheorem{claim}{Claim}

\newenvironment{sketch}{\hspace{-0.65 cm} \textbf{Proof Sketch.}}

\providecommand{\keywords}[1]
{
  \small	
  \textbf{\textit{Keywords---}} #1
}

\def\centerarc[#1](#2)(#3:#4:#5)
    { \draw[#1] ($(#2)+({#5*cos(#3)},{#5*sin(#3)})$) arc (#3:#4:#5); }

\tikzset{
    cross/.pic = {
    \draw[rotate = 45] (-#1,0) -- (#1,0);
    \draw[rotate = 45] (0,-#1) -- (0, #1);
    }
}

\usepackage{titling}
\predate{}
\postdate{}

\title{On Oriented Diameter of Power Graphs}
\date{} 

 \author{
Deepu Benson \\
 IIT Gandhinagar, India \\
    \texttt{bensondeepu@gmail.com} \\
  \and
  Bireswar Das \\
  IIT Gandhinagar, India \\
 \texttt{bireswar@iitgn.ac.in} \\
  \and
  Dipan Dey \\
  IIT Gandhinagar, India \\
  \texttt{dey\_dipan@iitgn.ac.in} \\
  \and
    Jinia Ghosh \\
  IIT Gandhinagar, India \\
  \texttt{jiniag@iitgn.ac.in} \\  
}
\everypar{\looseness=-1}

\begin{document}

\maketitle


\begin{abstract}
In this paper, we study the oriented diameter of power graphs of groups. We show that a $2$-edge connected power graph of a finite group has oriented diameter at most $4$. We prove that the power graph of the cyclic group of order $n$ has oriented diameter $2$ for all $n\neq 1,2,4,6$. For non-cyclic finite nilpotent groups, we show that the oriented diameter of corresponding power graphs is at least $3$. Moreover, we provide necessary and sufficient conditions for the oriented diameter of $2$-edge connected power graphs of finite non-cyclic nilpotent groups to be either $3$ or $4$. This, in turn, gives an algorithm for computing the oriented diameter of the power graph of a given nilpotent group that runs in time polynomial in the size of the group.
\end{abstract}

\section{ Introduction}

An \emph{orientation} $\OO$  of an undirected graph $X$ is an assignment of exactly one direction to each of the edges of $X$. An orientation is called a \emph{strong orientation} if any two vertices are reachable from each other by directed paths introduced by the orientation. It is easy to see that a graph with a bridge cannot admit a strong orientation. In 1939, Robbins \cite{robbins1939theorem} proved that a graph is strongly orientable if and only if it is $2$-edge connected\footnote{A graph is $2$-edge connected if and only if it is bridgeless and connected.}.

The \emph{diameter} of an undirected graph is the maximum distance between any two vertices in the graph. We denote the class of $2$-edge connected undirected graphs with diameter $d$ by $\mathscr{F}_d$.
For a directed graph $\mathfrak{X}$, the distance $d_\mathfrak{X}(u,v)$ of a vertex $v$ from a vertex $u$ is the length of a shortest directed path from $u$ to $v$. The \emph{diameter} of a directed graph $\mathfrak{X}$, denoted by $diam(\mathfrak{X})$, is the number $\max_{u,v}d_\mathfrak{X}(u,v)$. We write $diam(\mathfrak{X}):=\infty$ if there is no directed path from $u$ to $v$ for some pair of vertices $u$ and $v$ in $\mathfrak{X}$. Let $X_\OO$ be the directed graph obtained from $X$ after introducing the orientation $\OO$. The \emph{oriented diameter} $OD(X)$ of $X$ is defined to be the minimum number in the set $\{diam(X_\OO) \ | \ \OO \textrm{ is an orientation of } X\}$. Let $OD(\mathscr{F}_d):=max \{OD(X) \ | \ X\in \mathscr{F}_d\}$. Note that $OD(X)=\infty$ if the graph $X$ is not $2$-edge connected. For a graph with a single vertex, we assume both the diameter and the oriented diameter to be $0$.

While Robbins \cite{robbins1939theorem} provided the necessary and sufficient condition for the existence of a strong orientation of a graph, the paper does not offer any quantitative analysis of the difference in distances between a pair of vertices before and after strongly orienting the graph. In 1978, Chv{\'a}tal and Thomassen \cite{chvatal1978distances} accepted this challenge and proved that $\frac{1}{2} d^2 + d \leq OD(X) \leq 2d^2+2d$ for all $X\in \mathscr{F}_d$. The upper bound was improved by Babu et al. \cite{babu2021improvement} subsequently. 

The exploration of oriented diameters for classes of graphs with small values of diameter, as well as specific graph classes, was prompted by the quadratic upper bound on the oriented diameter. AT-free graphs \cite{fomin2004free} and chordal graphs \cite{fomin2004complexity} are popular such graph classes investigated. Attempts were also made to improve the general bound for $OD(\mathscr{F}_d)$ provided by Chv{\'a}tal and Thomassen \cite{chvatal1978distances} for specific values of $d$. From a result in \cite{fomin2004complexity}, it can be seen that $OD(\mathscr{F}_1) = 3$. Chv{\'a}tal and Thomassen \cite{chvatal1978distances} proved that $OD(\mathscr{F}_2) = 6$. A tight bound was obtained for $OD(\mathscr{F}_3)$ also. The results from \cite{kwok2010oriented, WANG2022374} proved that $OD(\mathscr{F}_3) = 9$. However, exact bounds are not available when $d > 3$. The current best upper bound is $21$, and the lower bound is $12$ for $OD(\mathscr{F}_4)$  \cite{babu2021improvement, chvatal1978distances}. The upper and lower bounds for $OD(\mathscr{F}_d)$ when $d \geq 5$ also follow from these two works. Moreover, these results demonstrate the challenging nature of determining the oriented diameter for classes of graphs, even when the diameter is very small.

There are several classes of graphs defined in terms of groups, e.g., Cayley graphs, commuting graphs, power graphs, etc. Cameron's survey contains an interesting collection of results on such graphs \cite{cameron2022graphs}.  In this paper, we focus on power graphs of finite groups (\Cref{def: power graph}), which were defined by Chakrabarty et al. \cite{chakrabarty2009undirected}.
 Abawajy et al. \cite{abawajy2013power} and Kumar et al. \cite{kumar2021recent} gave surveys on power graphs. 

Our primary motivation was to investigate if the symmetry structure of the underlying group of a power graph is useful for studying its oriented diameter. In this paper, we provide strong evidence that the algebraic structure is indeed helpful.

The diameter of any $2$-edge connected finite power graph $Pow(G)$ is at most $2$, since the identity element of $G$ is adjacent to every other vertex of the graph. Moreover, $Pow(G)$ is a complete graph if and only if $G$ is a cyclic group of prime power order \cite{chakrabarty2009undirected}. Hence, the results of Chv{\'a}tal et al. \cite{chvatal1978distances}, and Fomin et al. \cite{fomin2004complexity} imply that $OD(Pow(G))\leq 6$ for all $2$-edge connected power graphs. We obtain a tighter upper bound for power graphs by showing that every $2$-edge connected power graph has oriented diameter at most $4$ (\Cref{4 upper bound}). Moreover, the condition of $Pow(G)$ being $2$-edge connected simply translates to $G$ not having any maximal cyclic subgroup of order $2$.

Since power graphs have diameter at most 2, it is an interesting question to identify the classes of groups whose power graphs have oriented diameter 2. We first focus on the power graphs of finite cyclic groups.

For a graph to have an oriented diameter $2$, Czabarka et al. \cite{czabarka2019degree} and Cochran et al. \cite{cochran2021size} gave sufficient conditions on the minimum degree and the number of edges, respectively. Czabarka et al. \cite{czabarka2019degree} showed that if the minimum degree of a graph of order $n$ is at least $\frac{n}{2} + \frac{\ln n}{\ln(4/3)}$, then the graph has oriented diameter $2$. Whereas Cochran et al. \cite{cochran2021size} showed that if a graph of order $n$ has at least $\binom{n}{2}-n+5$ edges, then the graph has oriented diameter $2$. However, these conditions are \emph{not} satisfied by infinitely many power graphs of cyclic groups \footnote{In particular, for $n=2.3.5.k$ (or $n=3.5.7.11.13.k$), where $k$ is a squarefree natural number such that $gcd(k,2.3.5)=1$ (or $gcd(k,3.5.7.11.13)=1$), the power graph $Pow(\ZZ_n)$ does not satisfy the degree condition and the edge set size condition. 
See \Cref{app_degree_size_condition}}. We show that the oriented diameter of the power graph of a finite cyclic group of order $n$ is $2$, except $n=1,2,4,6$ (\Cref{cyclic_main}).

Nilpotent groups are important classes of groups that have been studied extensively (see, e.g., \cite{hall2018theory}). Some interesting results regarding the power graphs of finite nilpotent groups can be found in \cite{bera2022line,panda2023minimum}. We show that the oriented diameter of finite non-cyclic nilpotent groups is either $3$ or $4$. Moreover, we determine the exact conditions under which the oriented diameter is 3 and 4. Our main result in this paper is a complete group theoretic characterization of the oriented diameter of power graphs of nilpotent groups (\Cref{nil: main result}). We give this characterisation in terms of the uniqueness of certain subgroups and the existence of a certain maximal cyclic subgroup.

Next, we focus on the computational problem of computing the oriented diameter of a given graph $X$.  A key result by Chv{\'a}tal et al. \cite{chvatal1978distances} showed that it is NP-hard to decide whether a given undirected graph has oriented diameter $2$. This leads to the investigation of several versions of the problem by restricting the class of graphs. For computing orientations of AT-free graphs and chordal graphs, approximation algorithms are provided by Fomin et al. \cite{fomin2004free} and Fomin, Matamala and Rapaport \cite{fomin2004complexity} respectively. Eggemann and Noble \cite{eggemann2009minimizing}  designed a fixed-parameter tractable (FPT) algorithm that decides if a planar graph $X$ has oriented diameter at most $l$, where $l$ is the parameter. 

We show that the oriented diameter of the power graphs of nilpotent groups can be computed in polynomial time. It turns out it is rather straightforward to check the conditions in the characterization of the oriented diameter of power graphs of finite nilpotent groups in polynomial time. 

Our results on the oriented diameter of power graphs hinge on figuring out interesting combinatorial and algebraic structures of the power graphs. For example, the results on the power graphs of cyclic groups depend on a careful ``decomposition'' of the graph in ``layers'' using its subgroup structures, which in turn helps us to apply an inductive approach for constructing a diameter $2$ orientation (see \Cref{section: cyclic}). 

The orientations we construct in this paper depend on careful designs of \textit{gadgets} ($P_4$-gadget in \Cref{section: cyclic} and $C_4$-gadget in  \Cref{section: nilpotent }) and their placements in $Pow(G)$ using group theoretic properties (
\Cref{cyclic: tool for incrementing}, 
\Cref{nil: two odd primes}).
For a nilpotent group  $G$, we prove that for $Pow(G)$ to have oriented diameter $3$, the oriented edges of $Pow(G)$ must obey certain uniformity conditions (\Cref{Fat Arrow Lemma}). While proving an important lower bound on $OD(Pow(G))$ for nilpotent group $G$, these conditions are crucial for cutting down the number of possibilities of orienting edges in $Pow(G)$ (\Cref{nil: OD strictly greater than 3}).

 \section{Preliminaries}
 \label{preli}
For a simple graph $X=(V,E)$, the vertex set of $X$ is denoted by $V(X)$, and the edge set of $X$ is denoted by $E(X)$. For basic definitions and notations from graph theory, an interested reader can refer to any standard textbook (e.g., \cite{west2001introduction}). The \textit{induced subgraph} of $X$ on $S\subseteq V(X)$ is denoted by $X[S]$.  We denote a \textit{path} (both directed and undirected) from $u_1$ to $u_k$ by the sequence of vertices $u_1u_2\dots u_k$. A vertex $u$ is said to be a \textit{dominating vertex} of a graph $X$ if it is adjacent to every other vertex of $V(X)$. If $S,T \subseteq V(X)$, then $E(S,T)$ denotes the set of edges $\{s,t\}\in E(X)$, i.e., the set of edges with one endpoint from $S$ and another endpoint from $T$.

\begin{definition}\label{def: partial orientation}
    Let $X=(V,E)$ be an undirected graph. A subset $\OO \subseteq V\times V$ is said to be a partial orientation of $X$ if $\OO$ is obtained from assigning exactly one direction to a subset $E'$ of the edge set $E$. That is, for all $\{u,v\} \in E'$, either $(u,v)$ or $(v,u)$ is in $\OO$.
    We use $X_{\OO}$ to denote the directed graph $(V,\OO)$. Further, we denote the distance from a vertex $x$ to a vertex $y$ in the directed graph $X_{\OO}$ by $d_{X_{\OO}}(x,y)$.
\end{definition}

\begin{observation}\label{obs od is at most diam of partial orient}
    If $\OO$ is a partial orientation of an undirected graph $X$, then $OD(X) \leq diam(X_{\OO}) $.
\end{observation}

 The basic definitions and facts on group theory can be found in any standard book (e.g.,  \cite{rotman2012introduction}). In this paper, we only consider \emph{finite groups}. A subset $H$ of a group $G$ is called a \textit{subgroup} of $G$ if $H$ forms a group under the binary operation of $G$. This is denoted by $H \leq G$.

The number of elements in a group $G$ is called the \textit{order of the group}, denoted by $|G|$. The \textit{order of an element} $g$ in $G$, denoted by $o(g)$, is the smallest positive integer $m$ such that $g^m=e$, where $e$ is the identity element. A group $G$ is called \emph{cyclic} if $G=\{g, g^2, \dots, g^{m-1}, g^m=e\}$ for some $g\in G$. The element $g$ is called a generator of $G$, and we write $G=\langle g \rangle$. The set of all generators of a cyclic group $G$ is denoted by $gen(G)$.
For a cyclic group $G$, $|gen(G)|=\phi(|G|)$, where $\phi$ is the Euler's totient function. Recall that $\phi(p_1^{\al_1}\dots p_k^{\al_k})=p_1^{\al_1-1}(p_1-1)\dots p_k^{\al_k-1}(p_k-1)$, where $p_i$'s are distinct primes and $\al_i$'s are natural numbers. We call a cyclic subgroup $C$ of $G$ a \textit{maximal cyclic subgroup} of $G$ if $C$ is not properly contained in any cyclic subgroup of $G$. We use the following well-known group theoretic fact extensively in this paper.

\begin{fact}
\label{gtf_CLT}
 A finite cyclic group of order $n$ has a unique subgroup (which is also cyclic) of order $d$ for each divisor $d$ of $n$. 
\end{fact}

A group $G$ is called a \textit{$p$-group} if the order of each non-identity element is some positive power of a prime $p$. We denote the class of groups with prime power order by $\Gprime$. 
If $p^m$ is the highest power of a prime $p$ such that $p^m$ divides $|G|$, then a subgroup $H \leq G$ such that $|H|=p^m$ is called a \textit{Sylow p-subgroup} of $G$. The \textit{direct product} of two groups $G$ and $H$, denoted by $G\times H$, is the group with elements  $(g,h)$ where  $g \in G$ and $h \in H$ under the group operation $(g_1,h_1)(g_2,h_2)=(g_1g_2,h_1h_2)$, where the co-ordinate wise operations are the group operations of $G$ and $H$ respectively. A finite group is called a \textit{nilpotent group} if it is a direct product of its Sylow subgroups. Moreover, each Sylow subgroup is unique in a finite nilpotent group.

We now give the definition of power graphs (see \cite{cameron2022graphs}). 

\begin{definition}\label{def: power graph}
The power graph of a group $G$, denoted by $Pow(G)$, is an undirected graph with vertex set $G$, and edge set $E=  \{\{x,y\}: y = x^m$ \textrm{ for some integer} m \}.   
\end{definition}

 \begin{remark}\label{remark power graph}
     If $\{x,y\}$ is an edge in $Pow(G)$, then either $o(x)|o(y)$ or $o(y)|o(x)$.
 \end{remark}
 
 We define an equivalence relation $\sim$ on $G$ as follows: for $x,y\in G$, $x\sim y$ if and only if $\langle x \rangle = \langle y \rangle$, i.e., $x$ and $y$ generate the same cyclic subgroup of $G$. We call this equivalence class \textit{generator equivalence} class (in short, \textit{GE-class}). Let us denote the GE-class containing $x$ under $\sim$ by $[x]$. Note that $[x]=gen(\ang{x})$.
 So, all the elements of a GE-class are of the same order. We define the \emph{order of a GE-class} by the order of any element belonging to the class. 
   One can easily notice that the size of a class $[x]$ is $\phi(o(x))$.

\begin{remark}\label{equivalence classes form a clique} In $Pow(G)$ the following two facts hold:
    (\textit{$i$}) Each GE-class $[x]$ of $G$ induces a complete subgraph of $Pow(G)$;
    (\textit{$ii$}) For two GE-classes $[x]$ and $[y]$, if an element $x\in [x]$ is adjacent to an element $y\in [y]$ in $Pow(G)$, then every element of $[x]$ is adjacent to every element of $[y]$. Hence, in this case, it makes sense to say that $[x]$ and $[y]$ are adjacent in the graph $Pow(G)$.\\
    This remark motivates us to formulate the following definition.
\end{remark}

 \begin{definition}\label{adjacency of classes}
   Two distinct GE-classes $[x]$ and $[y]$ are called \textit{adjacent} if $x$ and $y$ are adjacent in $Pow(G)$.
 \end{definition}
In \Cref{ext preli}, we have provided an extended preliminary.

\section{Oriented Diameter of Power Graphs}

We begin the section by stating a necessary and sufficient condition on a finite group for the existence of a strong orientation of the corresponding power graph. The main result of this section is that the oriented diameter of $2$-edge connected power graphs is at most $4$. Fomin, Matamala and Rapaport \cite{fomin2004complexity} proved the following theorem about the oriented diameter of complete graphs, that is required for our further discussion. 

\begin{theorem} {\cite{fomin2004complexity}} \label{extension}
For every $n \geq 3$, $OD(K_n)=2$ except $n = 4$, and $OD(K_4)=3$. Moreover, for every $n \geq 5$, 
every strong orientation of $K_n$ with diameter $2$ can be extended to a strong orientation of $K_{n+1}$ with diameter $2$  and this extension can be constructed in linear time.
\end{theorem}

\begin{lemma}\label{power graph 2-edge connected}
    A power graph is $2$-edge connected if and only if the underlying group has no maximal cyclic subgroup of order $2$.
\end{lemma}

\begin{proof}
    If a group $G$ has a maximal cyclic subgroup $\langle g \rangle$ of order $2$, then $Pow(G)$ has a pendant vertex (a vertex with degree $1$) $g$ adjacent to the identity $e$, i.e., $\{e,g\}$ is a bridge in $Pow(G)$.

    For other direction, let $\{u,v\}$ be a bridge of $Pow(G)$. If none of $u$ and $v$ are identity, then the subgraph induced on $\{u,v,e\}$ forms a cycle, which is a contradiction to the fact $\{u,v\}$ is a bridge in $Pow(G)$. We can assume without loss of generality that $u=e$. If $v$ has a neighbour, say $v'$, in the graph $Pow(G)$, then $\{u=e,v,v'\}$ makes a cycle, and this again leads to the contradiction that $\{u=e,v\}$ is a bridge. So, $e$ is the only neighbour of $v$ in $Pow(G)$. This means $\langle v \rangle=\{v,e\}$ and $v \notin \langle g \rangle$ for any $g \in G \setminus \{e\}$. So, $\langle v \rangle$ is a maximal cyclic subgroup of order $2$.
\end{proof}

According to \Cref{power graph 2-edge connected}, the power graph of $\mathbb{Z}_2$, dihedral group $D_{2n}$  are not $2$-edge connected and hence are not strongly orientable.

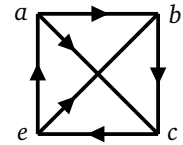
\begin{wrapfigure}{r}{0.55\textwidth}
\vspace{-0.5 cm}
\centering
\begin{tikzpicture}[scale=0.8]
\coordinate (e) at (0,0);
\coordinate (a) at (0,2);
\coordinate (b) at (2,2);
\coordinate (c) at (2,0);

\draw (e) node[left][scale=1] {$e$};
\draw (a) node[left][scale=1] {$a$};
\draw (b) node[right][scale=1] {$b$};
\draw (c) node[right][scale=1] {$c$};

\draw[ line width=0.5mm,black,opacity=1] (e) -- (a)  node[ currarrow, pos=0.5, xscale=1, sloped, scale=1] {};
\draw[ line width=0.5mm,black,opacity=1] (e) -- (b)  node[ currarrow, pos=0.25, xscale=1, sloped, scale=1] {};
\draw[ line width=0.5mm,black,opacity=1] (e) -- (c)  node[ currarrow, pos=0.5, xscale=-1, sloped, scale=1] {};
\draw[ line width=0.5mm,black,opacity=1] (a) -- (b)  node[ currarrow, pos=0.5, xscale=1, sloped, scale=1] {};
\draw[ line width=0.5mm,black,opacity=1] (a) -- (c)  node[ currarrow, pos=0.25, xscale=1, sloped, scale=1] {};
\draw[ line width=0.5mm,black,opacity=1] (b) -- (c)  node[ currarrow, pos=0.5, xscale=1, sloped, scale=1] {};

\end{tikzpicture}
\caption{An orientation of $K_4$ with $ecc(e)=2$}
\label{Fig: Orientation of $K_4$ with $ecc(e)=2$}
\end{wrapfigure} 

\begin{lemma}\label{glued graph}
    Let $X=(V, E)$ be an undirected graph with a dominating vertex $e$. Suppose $V\setminus \{e\}$ can be partitioned into sets $C_1,\ldots, C_m$ such that each induced subgraph $X[C_i]$ is a complete subgraph with at least two vertices, then the oriented diameter of $X$ is at most $4$. 
\end{lemma}

\noindent \textit{Proof.}  We claim that there is a partial orientation $\OO$ of the given graph $X$ such that the eccentricity  \footnote{The \emph{out-eccentricity} of a vertex $v$ of a directed graph $\XX$ is the maximum distance from $v$ to a vertex $u$ in $\XX$. The \emph{in-eccentricity} of a vertex $v$ of a directed graph $\XX$ is the maximum distance from a vertex $u$ in $\XX$ to $v$. The \emph{eccentricity} of a vertex $v$ of $\XX$ is the maximum of its out-eccentricity and in-eccentricity.} of $e$ in $\XO$ is $2$. This will give us $diam(\XO) \leq 4$, which in turn will imply that $OD(X) \leq 4$ (due to \Cref{obs od is at most diam of partial orient}). 
     Therefore, it is enough to give a partial orientation of each induced subgraph $X[C_i \cup \{ e \}]$,  such that the vertex $e$ has eccentricity 2 in the oriented subgraph $X[C_i \cup \{ e \}]$.

    We observe that for each $i$, $C_i\cup\{e\}$ induces a complete subgraph of $X$ of size at least 3. If $X[C_i \cup \{ e \}]$ is a complete subgraph of size $n \neq 4$, then by \Cref{extension}, we can orient the subgraph with diameter 2. In particular, $e$ has eccentricity 2 with this orientation. 

    Otherwise, if $C_i=\{ a,b,c \}$ then we can give an orientation to the induced subgraph $X[C_i \cup \{ e \}]$ (as shown in \Cref{Fig: Orientation of $K_4$ with $ecc(e)=2$}) with  $e$  having eccentricity 2 in the oriented subgraph $X[C_i \cup \{ e\}]$. \hfill$\square$

\begin{theorem}\label{4 upper bound}
    The oriented diameter of $Pow(G)$ is at most $4$, where $G$ is a finite group with no maximal cyclic subgroup of order $2$.
\end{theorem}
\begin{proof}
Let $S=G\setminus\{e\}$. Our idea is to partition $S$ into sets $C_1,\dots, C_m$ such that the condition of \Cref{glued graph} is satisfied.
To construct $C_1$, we pick a vertex $g \in S$ such that $o(g) >2$. Such a vertex exists as $G$ does not have any maximal cyclic subgroup of order $2$. Let $C_1=[g]$. Inductively, assume that we have constructed $C_1, \dots, C_l$. We pick a vertex $g$ in $S \setminus (\cup_{i=1}^{l}C_i)$ such that $o(g)>2$. The process ends if there is no such element. Otherwise, let $C_{l+1}=[g]$. 

Let  $C_1, \dots, C_m$ be the sets created at the end of the process. If $S \setminus (C_1 \cup \dots \cup C_m)$ is non-empty, it consists of elements of order $2$ only. Let $y \in S \setminus (C_1 \cup \dots \cup C_m) $. Since $\langle y \rangle$ is not a maximal cyclic subgroup, $y$ must be generated by some element $g$ of order more than $2$.  Let $g \in C_i$. Note that no other element $y' \in S \setminus (C_1 \cup \dots \cup C_m)$ can be generated by any element in $C_i$. Otherwise, it implies that $\langle g \rangle$ contains two elements of order $2$, which contradicts \Cref{gtf_CLT}. Now, as $g$ generates $y$,  the GE-class $[g]$ is adjacent to $[y]=\{y\}$ (by \Cref{adjacency of classes}). Hence, by \Cref{equivalence classes form a clique}, $C_i \cup \{y\}$ induces a clique. We update $C_i$ by $C_i \cup \{y\}$. Thus, each $y \in S \setminus (C_1 \cup \dots \cup C_m)$ can be merged to a unique $C_j$. Now, we apply \Cref{glued graph} to conclude that $OD(Pow(G))\leq 4$. Note that if $S \setminus (C_1 \cup \dots \cup C_m)$ is empty, then we can directly apply \Cref{glued graph} to obtain the result.\end{proof}
\section{Oriented Diameter of Power Graphs of Cyclic Groups}\label{section: cyclic}

Each cyclic group of order $n$ is isomorphic to $\ZZ_n$, where  $\ZZ_n$ is the additive group of integers modulo $n$. 
 \Cref{dom of a power graph} tells that when $n \geq 3$, $Pow(\ZZ_n)$ has at least two dominating vertices. Using this, we prove that $Pow(\ZZ_n)$, where $n \geq 3$, can be given a partial orientation of diameter $3$ (see \Cref{Upper3}). 
 
  \captionsetup{font=small}
\begin{wrapfigure}{r}{0.63\textwidth}

\centering
\begin{tikzpicture}[scale=0.35]
                

  \filldraw[color=black!60, fill=black!5, very thick][rounded corners] (6,13) rectangle (16,15.5);
  \draw (11,17) node[below][scale=1] {$\ZZ_{n} \setminus \{d_1,d_2\}$};
  \filldraw [black] (10,14) circle(2pt);
  \draw (10,14) node[above][scale=1] {$u$};
  \filldraw [black] (12,14) circle(2pt);
  \draw (12,14) node[above][scale=1] {$v$};
 

 \filldraw [black] (9,11) circle(4pt);
 \draw (8.9,11) node[left][scale=1] {$d_1$};
 
 \filldraw [black] (13,11) circle(4pt);
 \draw (13.1,11) node[right][scale=1] {$d_2$};

\draw (9,11) -- (13,11) node[ currarrow, pos=0.5, xscale=1, sloped, scale=1] {} ;
\draw (10,14) -- (13,11) node[ currarrow, pos=0.5, xscale=-1, sloped, scale=1] {} ;
\draw (10,14) -- (9,11) node[ currarrow, pos=0.5, xscale=-1, sloped, scale=1] {} ;
\draw (9,11) -- (12,14) node[ currarrow, pos=0.5, xscale=-1, sloped, scale=1] {} ;
\draw (13,11) -- (12,14) node[ currarrow, pos=0.5, xscale=-1, sloped, scale=1] {} ;

                \end{tikzpicture}

\caption{A partial orientation of $Pow(\ZZ_n)$ with diameter $3$.} 
\label{fig: upper3}
\end{wrapfigure}
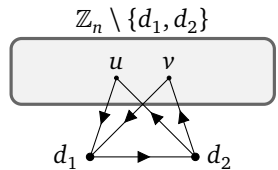
    
\begin{lemma}\label{dom of a power graph}\cite{cameron2010power, chakrabarty2009undirected}
    Let $G$ be a cyclic group. Then, the set of dominating vertices consists of all elements in $G$ if $G$ is of prime power order; otherwise, the set of dominating vertices is $gen(G)\cup\{e\}$.
\end{lemma}

\begin{lemma}
\label{Upper3}
    The oriented diameter of $Pow(\mathbb{Z}_n)$ is at most $3$, where $n \geq 3$.
\end{lemma}

\noindent \textit{Proof.} By \Cref{dom of a power graph}, $Pow(\mathbb{Z}_n)$ has at least two dominating vertices since $\mathbb{Z}_n$ has $\phi(n)\geq 2$ generators. Let $d_1$ and $d_2$ be two such dominating vertices. First we orient the edge $\{d_1,d_2\}$ as $(d_1,d_2)$. Then for any vertex $u \in \mathbb{Z}_n \setminus \{d_1,d_2\}$, we orient the edges $\{u,d_1\}$, $\{u,d_2\}$ as $(u,d_1)$ and $(d_2,u)$ respectively, forcing  $u,d_1,d_2$ to form a directed cycle (see \Cref{fig: upper3}). This gives an orientation with diameter $3$.
\hfill$\square$

\vspace{0.35cm}

Now, we present the main result on the oriented diameter of the power graphs of cyclic groups.

\begin{theorem}\label{cyclic_main}
  $
  \text{The oriented diameter of }Pow(\ZZ_n) =
    \begin{cases}
      0 & \text{if $n=1$}\\
      \infty & \text{if $n=2$}\\
      3 & \text{if $n=4,6$}\\
      2 & \text{otherwise}
    \end{cases}       
   $
\end{theorem}

The proof of \Cref{cyclic_main} requires a sequence of lemmata. The directed path of length $3$ plays an important role in those lemmata. Hence,  we call it a \emph{`$P_4$-gadget'} (the gadget is shown in \Cref{Fig:gadget}).

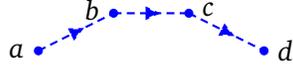
\begin{figure}[hpt!]
\centering
        \begin{tikzpicture}[scale=0.5]
        
        \coordinate (a) at (2,5.5);
        \coordinate (b) at (4,6.5);
        \coordinate (c) at (6,6.5);
        \coordinate (d) at (8,5.5);


        
        \draw (1.9,5.5) node[left][scale=1] {$a$};
        \draw (3.9,6.6) node[left][scale=1] {$b$};
        \draw (6.1,6.6) node[right][scale=1] {$c$};
        \draw (8.1,5.5) node[right][scale=1] {$d$};

        \draw[blue, densely dashed, thick] (a) -- (b) node[ currarrow, pos=0.5, xscale=1, sloped, scale=1, color=blue] {} ;
        \draw[blue, densely dashed, thick] (b) -- (c) node[ currarrow, pos=0.5, xscale=1, sloped, scale=1, color=blue] {} ;
        \draw[blue, densely dashed, thick] (c) -- (d) node[ currarrow, pos=0.5, xscale=1, sloped, scale=1, color=blue] {} ;

        \filldraw [blue] (a) circle(3pt);
        \filldraw [blue] (b) circle(3pt);
        \filldraw [blue] (c) circle(3pt);
        \filldraw [blue] (d) circle(3pt);

         \end{tikzpicture}
        \caption{The $P_4$-gadget}
         \label{Fig:gadget}
\end{figure}

\begin{lemma}\label{kn-lemma}
    For $n\geq 4$, there exists an optimal orientation of $K_n$ having a $P_4$-gadget as a subgraph.
\end{lemma}

\begin{proof} \Cref{Fig:strong-k4} shows an  optimal orientation of $K_4$ with a $P_4$-gadget. Figure~\ref{Fig:strong-k5} shows an optimal orientation of $K_5$ with a $P_4$-gadget. In both the figures, the subgraph formed by the directed edges $(a,b), (b, c)$, and $(c, d)$ gives the required $P_4$-gadget (marked in blue). By Theorem~\ref{extension}, this orientation of $K_5$ can be extended to obtain an orientation of $K_n$ with diameter $2$ for $n\geq 6$.\end{proof} 

\begin{figure}[hpt!]
\begin{subfigure}[b]{0.5\textwidth}
\centering
\hspace{0.9cm}
        \begin{tikzpicture}[scale=0.5]
        \coordinate (a) at (2,4.5);
        \coordinate (b) at (4,6.5);
        \coordinate (c) at (6,6.5);
        \coordinate (d) at (8,4.5);

        \draw (1.9,4.5) node[left][scale=1] {$a$};
        \draw (3.9,6.6) node[left][scale=1] {$b$};
        \draw (6.1,6.6) node[right][scale=1] {$c$};
        \draw (8.1,4.5) node[right][scale=1] {$d$};

        \draw[blue, densely dashed, thick] (a) -- (b) node[ currarrow, pos=0.5, xscale=1, sloped, scale=1.3, color=blue] {} ;
        \draw[blue, densely dashed, thick] (b) -- (c) node[ currarrow, pos=0.5, xscale=1, sloped, scale=1.3, color=blue] {} ;
        \draw[blue, densely dashed, thick] (c) -- (d) node[ currarrow, pos=0.5, xscale=1, sloped, scale=1.3, color=blue] {} ;
        
        \draw[black, thick] (a) -- (d) node[ currarrow, pos=0.5, xscale=-1, sloped, scale=1, color=black] {} ;
        \draw[black, thick] (b) -- (d) node[ currarrow, pos=0.75, xscale=-1, sloped, scale=1, color=black] {} ;
        \draw[black, thick] (a) -- (c) node[ currarrow, pos=0.35, xscale=-1, sloped, scale=1, color=black] {} ;

         \filldraw [blue] (a) circle(4pt);
         \filldraw [blue] (b) circle(4pt);
         \filldraw [blue] (c) circle(4pt);
         \filldraw [blue] (d) circle(4pt);

        \end{tikzpicture}
\caption{An orientation of $K_4$ with diameter $3$.}
\label{Fig:strong-k4}
\end{subfigure} \hspace{ 1 cm}
\begin{subfigure}[b]{0.4\textwidth}
\centering
       \begin{tikzpicture}[scale=0.6]

\coordinate (p) at (2,6);
\coordinate (q) at (4,4);
\coordinate (r) at (0,4); 
\coordinate (s) at (0.75,2);
\coordinate (t) at (3.25,2);

\draw (p) node[above] {$c$};
\draw (q) node[right]  {$e$};
\draw (r) node[left] {$b$};    
\draw (s) node[left] {$a$};
\draw (t) node[right] {$d$};   


\draw[black,thick] (p) -- (q)  node[ currarrow, pos=0.5, xscale=-1, sloped, scale=1] {};

\draw[blue,densely dashed, thick] (p) -- (r)  node[ currarrow, pos=0.5, xscale=1, sloped, scale=1.3, color=blue] {};

\draw[blue,densely dashed, thick] (s) -- (r)  node[ currarrow, pos=0.5, xscale=-1, sloped, scale=1.3, color=blue] {};

\draw[black,thick] (r) -- (q)  node[ currarrow, pos=0.5, xscale=1, sloped, scale=1] {};

\draw[black,thick] (p) -- (s)  node[ currarrow, pos=0.35, xscale=-1, sloped, scale=1] {};

\draw[blue,densely dashed, thick] (p) -- (t)  node[ currarrow, pos=0.35, xscale=1, sloped, scale=1.3, color=blue] {};

\draw[black,thick] (r) -- (t)  node[ currarrow, pos=0.5, xscale=1, sloped, scale=1] {};

\draw[black,thick] (q) -- (s)  node[ currarrow, pos=0.5, xscale=-1, sloped, scale=1] {};

\draw[black,thick] (t) -- (q)  node[ currarrow, pos=0.5, xscale=1, sloped, scale=1] {};

\draw[black,thick] (s) -- (t)  node[ currarrow, pos=0.5, xscale=-1, sloped, scale=1] {};
        \filldraw [blue] (p) circle(4pt);
        \filldraw [black] (q) circle(4pt);
        \filldraw [blue] (r) circle(4pt);
        \filldraw [blue] (s) circle(4pt);
        \filldraw [blue] (t) circle(4pt);

\end{tikzpicture}

\caption{An orientation of $K_5$ with diameter $2$.}
\label{Fig:strong-k5} 
\end{subfigure}

\caption{}
\label{Fig: k4,k5 with gadget}
\end{figure}

\begin{lemma}\label{general strategy for od 2}
    Let $X=(V, E)$ be an undirected graph such that $V=L_T \sqcup L_M \sqcup L_B$ (where $\sqcup$  denotes disjoint union) and the following properties hold: (a) There is a partial orientation of $X[L_T]$ with diameter at most $2$; (b) $|L_M|$ is even, $|L_M|\geq 4$, and $L_M$ is a set of dominating vertices of $X$; (c) There is a partial orientation $\OO_B$ of the edges of $X[L_B]$ and the edges in $E(L_T, L_B)$ such that there is a directed path of length at most $2$ between any two vertices $a,b \in L_B$ using only the directed edges in $\OO_B$. Then the oriented diameter of $X$ is $2$.
\end{lemma}

\begin{proof}
    We orient the graph $X$ with the following partial orientations $\OO_{\alpha}, \OO_{\beta}, \OO_{\gamma}$ (see \Cref{fig:Z_pq}).

    $\OO_{\alpha}$: Since the set $L_M$ induces a clique of size at least 4, by \Cref{kn-lemma} there is an optimal orientation of $X[L_M]$ having a $P_4$-gadget. In $\OO_{\alpha}$, we include this optimal orientation of $X[L_M]$ along with the optimal orientation of $X[L_T]$ (as per condition (a)) and $\OO_B$ (as per condition (c)).
         
    $\OO_{\beta}$: We first pick a $P_4$-gadget $(a,b,c,d)$ in $L_M$. Then, for any $u \in L_T$, we put $(u,a),(b,u),(u,c),(d,u)$ in $\OO_{\be}$. Also, depending on the directions of the edges $\{a,d\}$ and $\{b,c\}$ given in $\OO_{\al}$, we orient the edges between any vertex $r \in L_B$ and a vertex in $\{a,b,c,d\}$ such that $r,b,c$ as well as $r,a,d$ lie in a directed $3$-cycle. For example, if $(a,d) \in \OO_{\alpha}$, then we put $(r,a),(d,r)$ in $\OO_{\beta}$. See \Cref{fig:Z_pq}.
        
          \begin{figure}[hpt!]
         \label{fig:Z_pq}
\begin{subfigure}[b]{0.5\textwidth}
\centering
        \begin{tikzpicture}[scale=0.4]
        
        \coordinate (a) at (2,5.5);
        \coordinate (b) at (4,6.5);
        \coordinate (c) at (6,6.5);
        \coordinate (d) at (8,5.5);
        \coordinate (r) at (5,1);
        \coordinate (u) at (5,11);
        \coordinate (v) at (-1,6);
        \coordinate (w) at (1,6);
        
        \filldraw[color=teal!30, fill=teal!20, very thick][rounded corners] (-1.5,0) rectangle (11.5,2);
        \filldraw[color=teal!60, fill=teal!50, very thick][rounded corners] (-1.5,4) rectangle (11.5,7);
        \filldraw[color=teal!100, fill=teal!90, very thick][rounded corners] (-1.5,10) rectangle (11.5,12);

        \filldraw [black] (2.5,11) circle(5pt);
        \filldraw [black] (5,11) circle(5pt);
        \filldraw [black] (7.5,11) circle(5pt);
        \filldraw [black] (5,1) circle(5pt);
        \filldraw[black] (2.5,1) circle(5pt);
        \filldraw[black] (7.5,1) circle(5pt);
        \filldraw [black] (-1,6) circle(5pt);
        \filldraw [black] (1,6) circle(5pt);

        \draw (1.9,5.5) node[left][scale=0.8] {$a$};
        \draw (3.9,6.6) node[left][scale=0.8] {$b$};
        \draw (6.1,6.6) node[right][scale=0.8] {$c$};
        \draw (8.1,5.5) node[right][scale=0.8] {$d$};
        \draw (5,0.9) node[below][scale=0.8] {$r$};
        \draw (5,11.1) node[above][scale=0.8] {$u$};
        \draw (-1.1,6.1) node[above][scale=0.8] {$v$};
        \draw (1,6.1) node[above][scale=0.8] {$w$};
        
        \draw[blue, densely dashed] (a) -- (b) node[ currarrow, pos=0.5, xscale=1, sloped, scale=1, color=blue] {} ;
        \draw[ blue, densely dashed] (b) -- (c) node[ currarrow, pos=0.5, xscale=1, sloped, scale=1, color=blue] {} ;
        \draw[blue, densely dashed] (c) -- (d) node[ currarrow, pos=0.5, xscale=1, sloped, scale=1, color=blue] {} ;

       \draw (r) -- (b) node[ currarrow, pos=0.5, xscale=-1, sloped, scale=0.9] {} ;
       \draw (c) -- (r) node[ currarrow, pos=0.5, xscale=-1, sloped, scale=1] {} ; 
       \draw (b) -- (u) node[ currarrow, pos=0.5, xscale=1, sloped, scale=1] {} ; 
       \draw (u) -- (c) node[ currarrow, pos=0.5, xscale=1, sloped, scale=1] {} ; 
       \draw (u) -- (a) node[ currarrow, pos=0.5, xscale=-1, sloped, scale=1] {} ;
       \draw (u) -- (d) node[ currarrow, pos=0.5, xscale=-1, sloped, scale=1] {} ; 
         \draw (v) -- (w) node[ currarrow, pos=0.5, xscale=1, sloped, scale=1] {} ; 

         \draw[] (v) .. controls (2,10) and (4,11) .. (u) node[ currarrow, pos=0.35, xscale=-1, sloped, scale=1, color=black] {};
        \draw (w) -- (u) node[ currarrow, pos=0.5, xscale=1, sloped, scale=1] {} ; 

         \draw[] (v) .. controls (-2,4) and (3,2) .. (r) node[ currarrow, pos=0.35, xscale=-1, sloped, scale=1, color=black] {};
          \draw[] (w) .. controls (0,5) and (3,3) .. (r) node[ currarrow, pos=0.35, xscale=1, sloped, scale=1, color=black] {};

        


        \draw (12,11) node[right][scale=1] {$L_T$};
        \draw (12,6) node[right][scale=1] {$L_M$};
        \draw (12,1) node[right][scale=1] {$L_B$};

        \filldraw [blue] (a) circle(5pt);
        \filldraw [blue] (b) circle(5pt);
        \filldraw [blue] (c) circle(5pt);
        \filldraw [blue] (d) circle(5pt);

         \end{tikzpicture}
        \caption{ Partial orientations $\OO_{\al},\OO_{\be},\OO_{\gamma}$ when $|L_M| > 4$ }
        \label{fig:Z_pq_a}
    
\end{subfigure} \hspace{1 cm}
\begin{subfigure}[b]{0.43\textwidth}
\centering
          \begin{tikzpicture}[scale=0.4]
        
        \coordinate (a) at (2,4.5);
        \coordinate (b) at (4,6.5);
        \coordinate (c) at (6,6.5);
        \coordinate (d) at (8,4.5);
        \coordinate (r) at (5,1);
        \coordinate (u) at (5,11);
        \coordinate (v) at (-1,6);
        \coordinate (w) at (1,6);
        
         \filldraw[color=teal!30, fill=teal!20, very thick][rounded corners] (-1.5,0) rectangle (10,2);
        \filldraw[color=teal!60, fill=teal!50, very thick][rounded corners] (-1.5,4) rectangle (10,7);
        \filldraw[color=teal!100, fill=teal!90, very thick][rounded corners] (-1.5,10) rectangle (10,12);

        \filldraw [black] (2.5,11) circle(5pt);
        \filldraw [black] (5,11) circle(5pt);
        \filldraw [black] (7.5,11) circle(5pt);
        
         \filldraw [black] (2.5,1) circle(5pt);
        \filldraw [black] (5,1) circle(5pt);
        \filldraw [black] (7.5,1) circle(5pt);

        \draw (1.9,4.5) node[left][scale=0.8] {$a$};
        \draw (3.9,6.6) node[left][scale=0.8] {$b$};
        \draw (6.1,6.6) node[right][scale=0.8] {$c$};
        \draw (8.1,4.5) node[right][scale=0.8] {$d$};
        \draw (5,0.9) node[below][scale=0.8] {$r$};
        \draw (5,11.1) node[above][scale=0.8] {$u$};

        \draw[blue, densely dashed] (a) -- (b) node[ currarrow, pos=0.5, xscale=1, sloped, scale=1, color=blue] {} ;
        \draw[blue, densely dashed] (b) -- (c) node[ currarrow, pos=0.5, xscale=1, sloped, scale=1, color=blue] {} ;
        \draw[blue, densely dashed] (c) -- (d) node[ currarrow, pos=0.5, xscale=1, sloped, scale=1, color=blue] {} ;
        
        \draw[black, thick] (a) -- (d) node[ currarrow, pos=0.5, xscale=-1, sloped, scale=1, color=black] {} ;
        \draw[black, thick] (b) -- (d) node[ currarrow, pos=0.75, xscale=-1, sloped, scale=1, color=black] {} ;
        \draw[black, thick] (a) -- (c) node[ currarrow, pos=0.35, xscale=-1, sloped, scale=1, color=black] {} ;

       \draw (r) -- (b) node[ currarrow, pos=0.5, xscale=-1, sloped, scale=1] {} ;
       \draw (c) -- (r) node[ currarrow, pos=0.5, xscale=-1, sloped, scale=1] {} ; 
       \draw (b) -- (u) node[ currarrow, pos=0.5, xscale=1, sloped, scale=1] {} ; 
       \draw (u) -- (c) node[ currarrow, pos=0.5, xscale=1, sloped, scale=1] {} ; 
       \draw (u) -- (a) node[ currarrow, pos=0.5, xscale=-1, sloped, scale=1] {} ;
       \draw (u) -- (d) node[ currarrow, pos=0.5, xscale=-1, sloped, scale=1] {} ; 
       \draw (r) -- (a) node[ currarrow, pos=0.5, xscale=1, sloped, scale=1] {} ;
       \draw (r) -- (d) node[ currarrow, pos=0.5, xscale=1, sloped, scale=1] {} ;

        


        \draw (10.5,11) node[right][scale=1] {$L_T$};
        \draw (10.5,6) node[right][scale=1] {$L_M$};
        \draw (10.5,1) node[right][scale=1] {$L_B$};

        \filldraw [blue] (a) circle(5pt);
        \filldraw [blue] (b) circle(5pt);
        \filldraw [blue] (c) circle(5pt);
        \filldraw [blue] (d) circle(5pt);

         \end{tikzpicture}
         \caption{Partial orientations $\OO_{\al},\OO_{\be},\OO_{\gamma}$ when $|L_M|=$ $4$ }
         \label{fig:Z_pq_b}
\end{subfigure}
\caption{}
  \label{fig:Z_pq}
\end{figure}
     $\OO_{\gamma}$: When $|L_M|\neq 4$, partition the set $L_M\setminus \{a,b,c,d\}$ into disjoint pairs $\{v,w\}$. This partitioning is possible since $|L_M|$ is even. Now, we orient the edges between any vertex $r \in L_B$ and a vertex in $\{v,w\}$ such that $r,v,w$ lie in a directed $3$-cycle. For example, if $(v,w) \in \OO_{\alpha}$, then we put $(r,v),(w,r)$ in $\OO_{\gamma}$.

        The case when $|L_M|=4$ is slightly different and handled as shown in \Cref{fig:Z_pq_b}.

We now show that using $\OO_{\alpha},\OO_{\beta}$ and $\OO_{\gamma}$, we indeed get $OD(X)=2$. 
Let $X_\OO$ be the directed graph derived after orienting the edges of $X$ using the partial orientations $\OO_{\al},\OO_{\be},\OO_{\gamma}$. Note that using $\OO_{\al}$, there is a directed path of length at most $2$ between any two vertices of $L_T$ (and $L_B$). The same applies for $L_M$ if $|L_M|\geq 5$. Whereas, if $|L_M|=4$, then $\{a,d\}$ is the only pair of vertices in $L_M$ such that $\OO_{\al}$ gives a directed path of length $3$ from $a$ to $d$ in $X[L_M]$ (see \Cref{fig:Z_pq_b}). But since $(d,a)$ is in $\OO_{\al}$, we have put the directed edges $(a,r)$ and $(r,d)$ in $\OO_{\be}$ for any vertex $r\in L_B$. Hence, in this case, there is a directed path $ard$ of length $2$, which solves our purpose.
    
    From \Cref{fig:Z_pq} it is clear that $d_{X_{\OO}}(r,u)=d_{X_{\OO}}(u,r)=2$ for all vertices $u\in L_T$ and $r\in L_B$. Moreover, for all vertices $u\in L_T$, $y\in \{a,b,c,d\}\subset L_M$ and $r\in L_B$, we have $d_{X_{\OO}}(u,y)=d_{X_{\OO}}(y,u)=2$ and $d_{X_{\OO}}(r,y)=d_{X_{\OO}}(y,r)=2$. Now due to $\OO_{\gamma}$ every vertex $y$ of $L_M\setminus \{a,b,c,d\}$ participates in a directed $3$-cycle with any vertex $r$ of $L_B$ as well as with any vertex $u$ of $L_T$ (see \Cref{fig:Z_pq} ) and hence $d_{X_{\OO}}(u,y)=d_{X_{\OO}}(y,u)=2$ as well as $d_{X_{\OO}}(r,y)=d_{X_{\OO}}(y,r)=2$. 
    
    Hence, $diam(X_{\OO})=2$ and by \Cref{obs od is at most diam of partial orient}, we have $OD(X)=2$.      
\end{proof}

Let us state a useful fact about the structure of power graphs of cyclic groups.

\begin{fact}
\label{cyclic remark}
   Let $G$ be a cyclic group and $x,y \in G$. Then $\{x,y\}$ is an edge of $Pow(G)$ if and only if $o(x) | o(y)$ or $o(y)| o(x)$. Therefore $S$ is a clique in $Pow(G)$ if and only if $o(x)|o(y)$ or $o(y)|o(x)$ for all $x,y \in S$.
\end{fact}

\begin{lemma}\label{cyclic: primes 2 and q }
    If $q\geq 3$ is a prime, then the oriented diameter of $Pow(\mathbb{Z}_{2^{\al}q^{\be}})$, $\al , \be \geq 1$, is $2$ except when $(\al,\be,q)=(1,1,3)$ (i.e., for $\ZZ_{6}$).
\end{lemma}

\begin{proof}
    In this proof, we use the fact that a cyclic group $H$ has exactly $\phi(k)$ elements of order $k$ for each divisor $k$ of $|H|$. Let $G=\ZZ_{2^{\al}q^{\be}}$. Let $G_j$ be the subgroup of $G$ of order $2^{\al}q^{j}$, $1\leq j \leq \be$ (Since $G$ is cyclic, unique $G_j$ exits by \Cref{gtf_CLT}.). The idea is to inductively show that if $Pow(G_j)$ has oriented diameter $2$, so does $G_{j+1}$. For this, we apply \Cref{general strategy for od 2} with $L_B=G_j$, $L_M=gen(G_{j+1})=\{x | o(x)=2^{\al}q^{j+1}\}$, and $L_T=G_{j+1}\setminus (L_B\cup L_M)=\{x | \ o(x)=2^{k}q^{j+1}, \ 0\leq k \leq (\al-1)\}$. The proof is by induction on $j$. There are two base cases.

\textit{Base cases:}  
\begin{itemize}
     \item[1.] $(\al,q)\neq (1,3)$. Then, we use $j=1$ as the base case.\\
     We divide $G_1$ into three sets $L_B=\{x| \ o(x)=1 \text{ or } o(x)=2^k \cdot q \text{ where } 0 \leq k < \al \}$; $L_M= gen(G_1)=\{x| \ o(x)=2^{\al} \cdot q\}$; $L_T=\{x | \ o(x)=2^k \text{ where } 1 \leq k \leq \al \}$. Using \Cref{cyclic remark}, $L_B$ and $L_T $ induce complete subgraphs and, moreover, the corresponding induced subgraphs are isomorphic to  $K_{2^{\al-1}(q-1)+1}$ and $ K_{2^{\al}-1}$ respectively. $| L_M|= \phi(2^{\al}\cdot q)=2^{\al-1}(q-1)\geq 4$.

     \item[2.] $(\al,\be)=(1,3)$. Then, we use $j=2$ as the base case. \\
     We divide  $G_2$ into three sets $L_B=\{x| \ o(x)=2 \text{ or } 2 \cdot 3\}$; $L_M= gen(G_2)=\{x| \ o(x)=2 \cdot 3^{2}\}$; $L_T=\{x | \ o(x)=3^k \text{ where } 0\leq k \leq 2\}$. Using \Cref{cyclic remark}, $L_B$ and $L_T $ induce complete subgraphs and, moreover, the corresponding induced subgraphs are isomorphic to $ K_3$ and $ K_7$ respectively. $|L_M|= \phi(2 \cdot 3^{2})=6$.
     
 \end{itemize}

Now we verify that in both cases, the sets $L_B, L_M$ and $L_T$ satisfy the conditions of \Cref{general strategy for od 2}. Since, in the first case, $(\alpha,q) \neq (1,3)$, $|L_B|$ and $|L_T|$ are not equal to $2,4$ for any value of ${\al}$. So, in both cases, $L_B$ and $L_T$ are either singleton sets or induce complete subgraphs with oriented diameter $2$. Hence, it is sufficient to take $\OO_{B}$ as the optimal orientation of $X[L_B]$. Moreover, in each case, due to \Cref{dom of a power graph} $L_M$ consists of dominating vertices of $Pow(G_j)$, for $j=1,2$. Hence, by \Cref{general strategy for od 2}, the oriented diameter of $Pow(G_j)$, $j=1,2$, is $2$.
     
 \textit{Inductive step:} We assume that $OD(Pow(G_j))=2$ and want to show that $OD(Pow(G_{j+1}))=2$. For this, we divide $G_{j+1}$ into $L_B$, $L_M$ and $L_T$ as described in the proof sketch. Now using \Cref{cyclic remark} in $L_T$, any element of order $2^{k_1}q^{j+1}$ is adjacent to any element of order $2^{k_2}q^{j+1}$, where $0\leq k_1 < k_2 \leq (\al-1)$. Hence, $Pow(G_{j+1})[L_T]$ is a complete subgraph of size at least $\phi(q^2)\geq 6$ that can be oriented with diameter $2$. The set $L_M=gen(G_{j+1})$ contains dominating vertices of $Pow(G_{j+1})$. Moreover, as this is not the base case, $|L_M|=\phi(2^{\al}q^{j+1})\geq \phi(2^2 \cdot 3^2)=12$. Therefore, by \Cref{general strategy for od 2}, $OD(Pow(G_{j+1}))=2$.

Hence, by mathematical induction, $Pow(G_{\be})$ has oriented diameter $2$.
\end{proof}

We now state two group theoretic facts which are used in the proof of \Cref{cyclic: tool for incrementing} and \Cref{cyclic_main}. For a proof of \Cref{direct product of co-prime order}, one can refer to \Cref{app_direct product of co-prime order}.

\begin{fact}\label{direct product of co-prime order}
Let $G$ and $H$ be two finite groups such that $gcd(|G|,|H|)=1$. If $g_1$ generates $g_2$ in $G$ and $h_1$ generates $h_2$ in $H$, then $(g_1,h_1)$ generates $(g_2,h_2)$ in $G\times H$.  
\end{fact}

\begin{fact}\label{gtf_Z_mn}\cite{dummit2004abstract}
 If $m$ and $n$ are two relatively prime numbers, then $\ZZ_{mn} \cong \ZZ_{m}\times \ZZ_{n}$. 
\end{fact}

\begin{lemma}\label{cyclic: tool for incrementing}
    Let $H$ be a cyclic group such that $Pow(H)$ has oriented diameter $2$. If $gcd(|H|,p)=1$, where $p\neq 2$ is a prime, then the oriented diameter of $Pow(H\times \ZZ_{p^{\alpha}})$, $\alpha\geq 1$, is $2$.
\end{lemma}

First, we give a proof sketch of the lemma.

\vspace{0.1cm}

\begin{sketch}
Let $\Gamma=Pow(H\times \ZZ_{p^{\alpha}})$. We pick elements $g_0,\ldots, g_\alpha\in \ZZ_{p^\alpha}$ such that $o(g_i)=p^i$. This gives a tower of subgroups $\{e\}=\ang{g_0}\leq\ldots \leq \ang{g_\alpha}=\ZZ_{p^\alpha}$, where $e$ is the identity element of $\ZZ_{p^{\al}}$. Let $G_j=H\times \ang{g_j}$. Since $|H|$ and $|\ang{g_j}|$ are coprime to each other, by \Cref{gtf_Z_mn}, each $G_j$, $0 \leq j \leq \al$ is a cyclic subgroup of $H \times \ZZ_{p^{\al}}$. These subgroups form a tower of cyclic subgroups $G_0\leq \ldots \leq G_\alpha$. We note that $H\cong G_0$ and $G_\alpha=H\times \ZZ_{p^\alpha}$. By induction on $j$, we show that the induced subgraph $\Gamma_j=\Gamma[G_j]=Pow(G_j)$ has oriented diameter 2. 

As $\Gamma_0 \cong Pow(H)$, we have  $OD(\Gamma_0)=2$. For the inductive step, we use \Cref{general strategy for od 2}. Let  $L_T=G_{j-1}$. By the induction hypothesis, $\Gamma_j[L_T]=\Gamma_{j-1}$ has oriented diameter 2. Therefore, condition (a) of \Cref{general strategy for od 2} is satisfied. The set of generators of $G_j$ is $gen(H)\times [g_j]$. We pick $L_M$ to be the set of generators $gen(H)\times ([g_j]\setminus \{g_j\})$. Since $j>0$ and $p \neq 2$, $|[g_j]|=\phi(p^j)\geq 2$. Thus, $L_M \neq \emptyset$. We finally set $L_B=G_j\setminus (L_T\cup L_M)=((H\setminus gen(H))\times [g_j])\sqcup (gen(H)\times \{g_j\})$.  We show conditions (b) and (c) of \Cref{general strategy for od 2} in the main proof.
 \hfill $\lhd$

\end{sketch}

\vspace{0.1cm}

Now we go into more details of the proof.

 \begin{proof} We note that $|H|\notin\{2,4\}$ as $Pow(H)$ has oriented diameter 2. Moreover, if $|H|=3$ then $p\geq 5$.   
    
The set $L_M$ being a subset of generators of $G_j$ consists of dominating vertices of $\Gamma_j=Pow(G_j)$,  and $|L_M|=|gen(H)|\times|[g_j]\setminus \{g_j\}|$ is even since $|gen(H)|=\phi(|H|)$ is an even number (as $|H|\neq 2$).
 
Now we show conditions (b) and (c) of \Cref{general strategy for od 2}.

    The set $L_M$ being a subset of generators of $G_j$ consists of dominating vertices of $\Gamma_j=Pow(G_j)$,  and $|L_M|=|gen(H)|\times|[g_j]\setminus \{g_j\}|$ is even since $|gen(H)|=\phi(|H|)$ is an even number (as $|H|\neq 2$).
    
    For $L_M$ to satisfy the condition (b) of \Cref{general strategy for od 2}, $|L_M|$ should be greater than or equal to $4$. As $|H|\neq 2$, we have $|gen(H)| \geq 2$. But the situation when $|gen(H)|=2$ and $|[g_j]|=2$ is problematic since it yields $|L_M|=2$. Now $|[g_j]|=2$ happens only if $p=3$. But in that case, as $gcd(|H|,p)=1$ and $|H| \neq 2$ or $4$, $|H|$ must have a prime factor greater than or equal to $5$ or $|H|$ must be divisible by $2^3$. In that case, $|gen(H)| \geq 4$ and hence, $|L_M| \geq 4$.

    The rest of the proof involves showing that condition (c) of \Cref{general strategy for od 2} is satisfied, i.e., there exists an orientation $\OO_B$ of the edges of $\Gamma[L_B]$ and $E(L_T, L_B)$ such that there is a directed path of length at most $2$ between any two vertices using only the directed edges in $\OO_B$.

Observe that, $L_B =((H\setminus gen(H))\times [g_j])\sqcup (gen(H)\times \{g_j\}) \subseteq G_{j} \setminus G_{j-1}$. Let $\OO_H$ be an orientation of $H$  having diameter $2$. Our idea is to mimic the orientation $\OO_H$ of $H$ while being oblivious to the second component of a vertex in $L_B$. In other words, for pairs of vertices $(u,g)$ and $(v,g')$ in $L_B$, if $(u,v) \in \OO_H$ we put $((u,g),(v,g'))$ in $\OO_B$, else we put $( (v,g'), (u,g) )$ in $\OO_B$. Note that if $\{u,v\}$ is an edge in $Pow(H)$, then $\{(u,g), (v,g')\}$ is an edge in $\Gamma_j$ (This can be verified easily by using \Cref{direct product of co-prime order}.).

     Since there is a directed path of length at most $2$ between two distinct vertices $u$ and $v$ in $Pow(H)$, the newly added directed edges in $\OO_B$ imply a directed path of length at most $2$ between two distinct vertices $(u,g'_{j})$ and $(v,g''_{j})$, where $u \neq v$ and $g'_{j}$ may or may not be equal to $g''_{j}$. So, the only remaining case to handle is when $u=v$, i.e., when both the vertices are from $(H \setminus gen(H))\times [g_j]$. Now, observe that for all $u \in H \setminus gen(H)$, the set $\{u\} \times [g_j] \subseteq L_B$ is a clique (due to \Cref{direct product of co-prime order}). Now if $|[g_j]| \neq 2$, we put the optimal orientation of $ \Gamma[\{u\}\times [g_j]]$ (using \Cref{extension}) in $\OO_B$.

     Note that if $|[g_j]| \neq 2$ or $4$, for any $a,b \in \{u\}\times [g_j]$, for all $u \in H \setminus gen(H)$, we have a directed path of length at most $2$. If $|[g_j]|=2$ or $4$, there exist exactly two vertices $a=(u,g_j), b=(u,g'_j)$ in each $\{u\} \times [g_j]$ such that $d_{\Gamma_{\OO_{B}}}(a,b)=3$ (where $\Gamma_{\OO_B}$ is the directed graph $(V(\Gamma),\OO_B)$). To solve this, we use the edges $E(L_T,L_B)$. Let $e'$ be the identity element of $H\times \ZZ_{p^{\alpha}}$. Since $e'\in H\times \ang{g_0} = G_0 \subseteq G_{j-1}$, $e'$ is in $L_T$ and $e'$ is adjacent to all the vertices in $L_B$. Now for a fixed $u \in H \setminus gen(H)$, we orient the edges $\{a,e'\}$ and $\{b,e'\}$ (depending on whether $(a,b) \in \OO_B$ or $(b,a) \in \OO_B$) so that $a,b,e'$ create a directed triangle in $\Gamma_{\OO_{B}}$. 
     We do this for all $u \in H \setminus gen(H)$. This gives a directed path of length at most $2$ for all the remaining pairs of vertices from $L_B$. Hence, condition (c) of \Cref{general strategy for od 2} is satisfied. Now, we apply \Cref{general strategy for od 2} and get an orientation of $\Gamma_j$.    
 \end{proof}

Now we are ready to prove the main result (\Cref{cyclic_main}) of this section.

\vspace{0.2cm}

\noindent \textbf{Proof of \Cref{cyclic_main}.}
The cases when $n=1$ and $n=2$ are straightforward to observe. By \Cref{extension} and observing $Pow(\mathbb{Z}_4)=K_4$, we have $OD(Pow(\mathbb{Z}_4))=3$. We have proved the case for $n=6$ in \Cref{app_cyclic_Z6}. Using \Cref{dom of a power graph} and \Cref{extension}, we get $OD(Pow(G))=2$ for a cyclic group $G \in \Gprime$  that is not $\mathbb{Z}_2$ or $\mathbb{Z}_4$.

Now, we are left with the case when $n$ has at least two prime factors and $n\neq 6$. Let $n={p_1^{\alpha_1} p_2^{\alpha_2}p_3^{\alpha_3}... p_k^{\alpha_k}}$ be the prime factorization of $n$, where $p_i$'s are distinct primes, $\al_i$'s are positive powers and $k \geq 2$. By \Cref{gtf_Z_mn}, we can write $\mathbb{Z}_n= \prod_{i\in S} \mathbb{Z}_{p_i^{\alpha_i}} \times \prod_{i\notin S} \mathbb{Z}_{p_i^{\alpha_i}}$ for any $S\subseteq [n]$. One can check that by suitably picking a subset $S$ of size at most $2$, we can ensure that the oriented diameter of the power graph of  $H=\prod_{i\in S} \mathbb{Z}_{p_i^{\alpha_i}}$ is 2. In particular, we consider $p_1$ and $p_2$ the smallest and the largest prime, respectively. We take a recursive approach to achieve an orientation of $Pow(\mathbb{Z}_n)$ with diameter $2$. If $p_1=2$, then we start with orienting the power graph of $H=\ZZ_{2^{\al_1}{p_2}^{\al_2}}$ with diameter $2$ by applying \Cref{cyclic: primes 2 and q }. If $p_1>2$, then we start with orienting the power graph of $H=\ZZ_{p_1^{\al_1}}$ with diameter $2$. In the first case, we extend $H$ recursively by $\ZZ_{p_3}^{\al_3},\dots,\ZZ_{p_k}^{\al_k}$, and with $(k-2)$ applications of \Cref{cyclic: tool for incrementing}, we get $OD(Pow(\ZZ_n))=2$. Whereas, in the second case, we extend $H$ recursively by $\ZZ_{p_2}^{\al_2},\dots,\ZZ_{p_k}^{\al_k}$, and with $(k-1)$ applications of \Cref{cyclic: tool for incrementing}, we get $OD(Pow(\ZZ_n))=2$. \hfill $ \square$

\section{Oriented Diameter of Power Graphs of $p$-groups}\label{section: p-group}

In this section, we study the oriented diameter of power graphs for finite non-cyclic groups from the class $\Gprime$ (recall that $\Gprime =\{ G \ | \ G \text{ is a $p$-group for some prime $p$}\}$). The main result of this section is \Cref{od of all p-groups}, where we fully characterize the group class $\Gprime$.

The definition of generalized quaternion group $Q_{2^n}$ of order $2^n$ can be found in any standard textbook of abstract algebra (for example, see \cite{gorenstein1980finite}). We note that a generalized quaternion group of order $4n$ for any $n$ can be defined. In this paper, we just need quaternion groups of order $2^n$ and a few facts about such groups, which we list below.

\begin{fact}\label{fact: gen quat} \cite[Theorem 4.2]{conrad2014generalized}
   The generalized quaternion $Q_{2^n}, n\geq 3$ contains \footnote{Note that $Q_{2}\cong \ZZ_2$ and $Q_{2^2} \cong \ZZ_4$.} exactly one maximal cyclic subgroup $\langle x \rangle$ of order $2^{n-1}$, and each element outside $\langle x \rangle$ is of order $4$. 
\end{fact}

Moreover, we use the following two statements; one is a lemma by Burnside (\Cref{unique p-subgroup of a p-group}, 1911) and another one is a result from \cite{MR3200118} in the proof of the next theorem.

\begin{lemma}\label{unique p-subgroup of a p-group}\cite{burnside1911theory}
    Let $G$ be a $p$-group for a prime $p$, which is neither cyclic nor generalized quaternion. Then $G$ has at least two subgroups of order $p$.
\end{lemma}

\begin{lemma}\label{proper power graph}\cite[Corollary 2]{MR3200118}
    Let $G \in \Gprime$. Then $Pow(G)\setminus \{e\}$ is connected if and only if $G$ is either cyclic or generalized quaternion.
\end{lemma}

\begin{theorem}\label{od of all p-groups}
    (1) Let $G \in \Gprime$ be neither cyclic nor generalized quaternion. If $G$ has no maximal cyclic subgroup of order $2$, then the oriented diameter of $Pow(G)$ is $4$.\\ (2) The oriented diameter of $Pow(Q_{2^n})$ is $3$, where $Q_{2^n}, n\geq 3$ is the generalized quaternion group.
\end{theorem}

\begin{proof}

(1)
     Let $\Gamma= Pow(G)$. Due to \Cref{4 upper bound}, it is sufficient to prove that $OD(\Gamma)\geq 4$. By \Cref{proper power graph}, we know that $\Gamma \setminus \{e\}$ is disconnected and hence $\Gamma \setminus \{e\}$ has at least two connected (A \textit{connected component} of a graph is a maximal connected subgraph of the graph.) components $C_1$ and $C_2$. So, there is no undirected $e$-avoiding path \footnote{A path in $Pow(G)$ for a group $G$ is $e$-\textit{avoiding} if it does not include the vertex corresponding to the identity element $e$ of $G$.} between a vertex of $C_1$ and a vertex of $C_2$ in the graph $\Gamma$.

   Now we prove that, for any arbitrary orientation $\OO$ of $\Gamma$, there are two vertices $u_1\in C_1$ and $u_2 \in C_2$ such that $d_{\Gamma_{\OO}}(u_1,u_2) \geq 4$ (recall $\Gamma_{\OO}$ denotes the directed graph $(V(\Gamma),\OO)$). For that, let us consider two elements $c_1 \in C_1$ and $c_2 \in C_2$. Without loss of generality, let us assume that $(c_1 ,e) \in \OO$. If $(e,c_2) \in \OO$ then $d_{\GA}(c_2,e), d_{\GA}(e,c_1) \geq 2$. Thus, we have $d_{\GA}(c_2,c_1) \geq 4$. For the other case, suppose $(c_2,e) \in \OO$. Now, to have $d_{\GA}(c_1,c_2) \leq 3$, we must have a vertex $d$ in $C_2$ such that $(e,d), (d,c_2) \in \OO$. Analogously, to have $d_{\GA}(c_2,c_1) \leq 3$, we must have a vertex $d' \in C_1$ such that $(e,d'), (d',c_1) \in \OO$. This gives us $d_{\GA}(d',e), d_{\GA}(e,c_2) \geq 2$ which implies $d_{\GA}(d',c_2) \geq 4$. So, using $\OO$, there is no directed path of length at most $3$ from $d' \in C_1$ to $c_2 \in C_2$. 

\vspace{0.3cm}

\noindent (2) \Cref{fact: gen quat} and \Cref{gtf_CLT} implies that $Q_{2^n}$ has a unique subgroup, say $\ang{y}$, of order $2$ . Since any element in $Q_{2^n}\setminus \langle x \rangle$ (see description of $x$ in \Cref{fact: gen quat}) belongs to some maximal cyclic subgroup of order $4$, there are $\frac{2^n-2^{n-1}}{\phi(4)}=2^{n-2}\geq 2$ (as $n\geq 3$) maximal cyclic subgroups of order $4$ in $Q_{2^n}$. Hence $Q_{2^n}$ has at least two maximal cyclic subgroups $C_1, C_2$ of order $4$ and one cyclic subgroup $C_3$ of order $4$ inside $\langle x \rangle$ such that the intersection $C_i \cap C_j=\{e,y\}$, $1\leq i < j \leq 3$, where $y$ is the unique element of $Q_{2^n}$ of order $2$. Let the two elements of order $4$ in $C_i$ be $c_{i1}$ and $c_{i2}$ such that $1 \leq i \leq 3$. Since $C_i \cap C_j=\{e,y\}$ for $i \neq j$, a path between a vertex $c_{ir}$, $r=1,2$ and a vertex in $c_{js}$, $s=1,2$ in $Pow(Q_{2^n})$ has to include $e$ or $y$.

    \begin{figure}[hpt!]
\centering
        \begin{tikzpicture}[scale=0.4]
        
        \coordinate (a) at (3,10);
        \coordinate (b) at (7,10);
        \coordinate (c) at (3,2);
        \coordinate (d) at (7,2);
        \coordinate (e) at (4,6);
        \coordinate (y) at (6,6);

        \draw (3,10.1) node[above][scale=0.8] {$c_{11}$};
        \draw (7,10.1) node[above][scale=0.8] {$c_{12}$};
        \draw (3.9,6) node[left][scale=0.8] {$e$};
        \draw (6.1,6) node[right][scale=0.8] {$y$};
        \draw (7,1.9) node[below][scale=0.8] {$c_{22}$};
        \draw (3,1.9) node[below][scale=0.8] {$c_{21}$};

       \draw[blue, densely dashed, thick] (e) -- (a) node[ currarrow, pos=0.5, xscale=1, sloped, scale=1] {} ;
       \draw[] (e) -- (b) node[ currarrow, pos=0.5, xscale=-1, sloped, scale=1] {} ;
       \draw[blue, densely dashed, thick] (e) -- (c) node[ currarrow, pos=0.5, xscale=-1, sloped, scale=1] {} ;
       \draw[] (e) -- (d) node[ currarrow, pos=0.5, xscale=1, sloped, scale=1] {} ;
       \draw[] (y) -- (a) node[ currarrow, pos=0.5, xscale=-1, sloped, scale=1] {} ;
       \draw[] (y) -- (b) node[ currarrow, pos=0.5, xscale=1, sloped, scale=1] {} ;
       \draw[] (y) -- (c) node[ currarrow, pos=0.5, xscale=1, sloped, scale=1] {} ;
       \draw[] (y) -- (d) node[ currarrow, pos=0.5, xscale=-1, sloped, scale=1] {} ;

       \draw (a) -- (b) ; 
       \draw (c) -- (d) ; 
       \draw (e) -- (y) ;

       \draw (5,9) node[][scale=0.8] {$C_1$};
       \draw (5,3) node[][scale=0.8] {$C_2$};

        \filldraw [black] (a) circle(5pt);
        \filldraw [black] (b) circle(5pt);
        \filldraw [black] (c) circle(5pt);
        \filldraw [black] (d) circle(5pt);
        \filldraw [black] (e) circle(5pt);
        \filldraw [black] (y) circle(5pt);

         \end{tikzpicture}
        \caption{Choosing the directed path $c_{11}ec_{21}$ forces the directed edges shown in the figure.}
         \label{fig:generalised quaternion}
\end{figure}
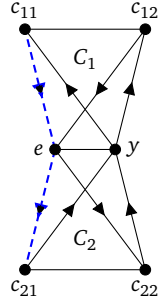
   
Let $\Gamma = Pow(Q_{2^n})$. First, we show that no orientation of $\Gamma$ has diameter $2$. For the sake of contradiction, we assume that $\OO$ is an orientation of $\Gamma$ such that the diameter of ${\Gamma_{\OO}}$ is $2$. The directed path from $c_{11}$ to $c_{21}$ of length $2$ in ${\Gamma_{\OO}}$ must include either $e$ or $y$. Without loss of generality, we assume that the path includes $e$. The case when the directed path from $c_{11}$ to $c_{21}$ of length $2$ in ${\Gamma_{\OO}}$ passes through $y$ can be dealt similarly.  Now, since the directed path from $c_{11}$ to $c_{21}$ of length $2$ in ${\Gamma_{\OO}}$ passes through $e$ (see \Cref{fig:generalised quaternion}), $(c_{11},e)$ and $ (e,c_{21})$ must be in $\OO$. In this case, the directed path from $c_{21}$ to $c_{11}$ of length $2$ has to include $y$. Hence, we must have $(c_{21},y), (y,c_{11}) \in \OO$. This also implies that $(c_{12},e), (e,c_{22}), (c_{22},y), (y,c_{12}) \in \OO$. Now, to have a directed path of length $2$ from $c_{11}$ to $c_{31}$, we need $(e,c_{31}) \in \OO$. On the other hand, to have a directed path of length $2$ from $c_{31}$ to $c_{21}$, we need $(c_{31},e) \in \OO$. This means we can have a directed path of length at most $2$ either from $c_{11}$ to $c_{31}$ or from $c_{31}$ to $c_{21}$, but not both. This contradicts the diameter of ${\Gamma_{\OO}}$ being $2$. 
   
Now, due to \Cref{obs od is at most diam of partial orient}, it is sufficient to give a partial orientation of $Pow(Q_{2^n})$ with diameter $3$. Such a partial orientation $\OO$ of $Pow(Q_{2^n})$ is shown in \Cref{fig: orientation of generalized quaternion}. In \Cref{fig: orientation of generalized quaternion}, $C_1, C_2, \dots, C_m$ denote the maximal cyclic subgroups of order $4$, where $m = 2^{n-2}$; $c_{i1}$ and $c_{i2}$ denote the elements of order $4$ in $C_i$. Note that $C_i \cap C_j=\{e,y\}$, for $1\leq i < j \leq m$ and $C_i \cap \langle x \rangle = \{e,y\}$. We partition the set $\langle x \rangle \setminus \{e,y\}$ into two arbitrary non-empty subsets $A$ and $B$ (note that $|\ang{x}| = 2^{n-1} \geq 4$). 
    
    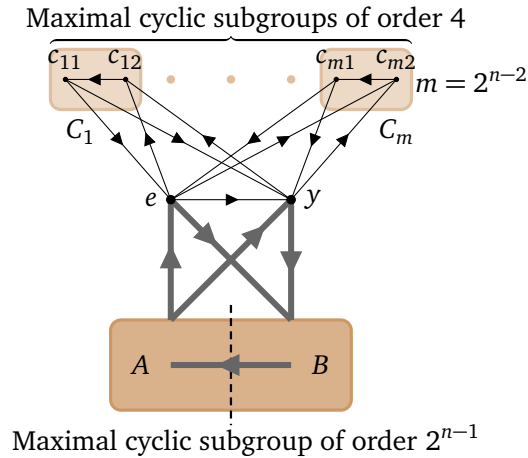
\begin{figure}[hpt!]

\centering
\begin{tikzpicture}[scale=0.4]
                


 \filldraw[color=brown!80, fill=brown!60, very thick][rounded corners] (7,2) rectangle (15,5);
 \draw (11.5,1) node[][scale=1] {Maximal cyclic subgroup of order $2^{n-1}$};
  \draw[line width=0.3 mm,densely dashed, black] (11,5.5) -- (11,1.5) ;
  \draw (8,3.5) node[][scale=1] {$A$};
  \draw (14,3.5) node[][scale=1] {$B$};

  \draw [line width=0.8mm,color=gray!120] (9,3.5) -- (13,3.5) node[ currarrow, pos=0.5, xscale=-1, sloped, scale=1.2] {} ;
  
  \draw [line width=0.8mm,color=gray!120] (9,5) -- (9,9) node[ currarrow, pos=0.5, xscale=1, sloped, scale=1.2] {} ;
  \draw [line width=0.8mm,color=gray!120] (9,5) -- (13,9) node[ currarrow, pos=0.7, xscale=1, sloped, scale=1.2] {} ;
  \draw [line width=0.8mm,color=gray!120] (9,9) -- (13,5) node[ currarrow, pos=0.3, xscale=1, sloped, scale=1.2] {} ;
  \draw [line width=0.8mm,color=gray!120] (13,9) -- (13,5) node[ currarrow, pos=0.5, xscale=1, sloped, scale=1.2] {} ;

  \filldraw[color=brown!50, fill=brown!30, very thick][rounded corners] (5,12) rectangle (8,14);
  \draw (6,12) node[below][scale=1] {$C_1$};
  \filldraw [black] (5.5,13) circle(2pt);
  \draw (5.5,13) node[above][scale=1] {$c_{11}$};
  \filldraw [black] (7.5,13) circle(2pt);
  \draw (7.5,13) node[above][scale=1] {$c_{12}$};
  
  \filldraw [color=brown!50, thick] (9,13) circle(3pt); 
  \filldraw [color=brown!50, thick] (11,13) circle(3pt);
  \filldraw [color=brown!50, thick] (13,13) circle(3pt);
  
 \filldraw[color=brown!50, fill=brown!30, very thick][rounded corners] (14,12) rectangle (17,14);
 \draw (16.5,12) node[below][scale=1] {$C_m$};
  \filldraw [black] (14.5,13) circle(2pt);
  \draw (14.5,13) node[above][scale=1] {$c_{m1}$};
  \filldraw [black] (16.5,13) circle(2pt);
  \draw (16.5,13) node[above][scale=1] {$c_{m2}$};
  \draw (19,13) node[][scale=1] {$m=2^{n-2}$};
 
 \draw (11.5,15) node[][scale=1] {Maximal cyclic subgroups of order $4$};
\draw [decorate,
    decoration = {calligraphic brace}][thick] (5,14.2) --  (17,14.2);
    
 \filldraw [black] (9,9) circle(4pt);
 \draw (8.9,9) node[left][scale=1] {$e$};
 
 \filldraw [black] (13,9) circle(4pt);
 \draw (13.1,9) node[right][scale=1] {$y$};

\draw (9,9) -- (13,9) node[ currarrow, pos=0.5, xscale=1, sloped, scale=1] {} ;
\draw (5.5,13) -- (13,9) node[ currarrow, pos=0.5, xscale=1, sloped, scale=1] {} ;
\draw (5.5,13) -- (9,9) node[ currarrow, pos=0.5, xscale=1, sloped, scale=1] {} ;
\draw (9,9) -- (7.5,13) node[ currarrow, pos=0.5, xscale=-1, sloped, scale=1] {} ;
\draw (13,9) -- (7.5,13) node[ currarrow, pos=0.5, xscale=-1, sloped, scale=1] {} ;
\draw (7.5,13) -- (5.5,13) node[ currarrow, pos=0.5, xscale=-1, sloped, scale=1] {} ;

\draw (14.5,13) -- (13,9) node[ currarrow, pos=0.5, xscale=-1, sloped, scale=1] {} ;
\draw (14.5,13) -- (9,9) node[ currarrow, pos=0.5, xscale=-1, sloped, scale=1] {} ;
\draw (9,9) -- (16.5,13) node[ currarrow, pos=0.5, xscale=1, sloped, scale=1] {} ;
\draw (13,9) -- (16.5,13) node[ currarrow, pos=0.5, xscale=1, sloped, scale=1] {} ;
\draw (16.5,13) -- (14.5,13) node[ currarrow, pos=0.5, xscale=-1, sloped, scale=1] {} ;





                \end{tikzpicture}

\caption[scale=0.7]{A partial orientation $\OO$ of $Pow(Q_{2^n})$ with diameter $3$. Here $y$ is the element of order $2$.} 
\label{fig: orientation of generalized quaternion}
\end{figure}

    We put $(e,y)$ in $\OO$. For all $a \in A$ and for all $b\in B$, we put the following directions in $\OO$: $(b,a)$, $(y,b)$, $(e,b)$, $(a,e)$, $(a,y)$. Moreover, we put the following directions in $\OO$: $(c_{i1},e)$, $(c_{i1},y)$, $(c_{i2},c_{i1})$, $(e, c_{i2})$, $(y, c_{i2})$, for each $i$, $1\leq i \leq 2^{n-2}$. From \Cref{fig: orientation of generalized quaternion}, it is easy to observe that the diameter of $Pow(Q_{2^n})$ using $\OO$ is $3$.
\end{proof}

\section{ Oriented Diameter of Power Graphs of Nilpotent Groups}\label{section: nilpotent }

 Since in the previous section we have dealt with non-cyclic finite groups of $\Gprime$, in this section, we only consider finite non-cyclic nilpotent groups $G$ such that $G \notin \Gprime$. We fully characterize the oriented diameter of power graphs of all such groups in the main result (see \Cref{nil: main result}) of this section. We write $\pi(G)$ to denote the set of all prime divisors of $|G|$. We start our discussion with a fact about finite nilpotent groups, which follows from two group theoretic facts: \Cref{co-prime orders commute} and \Cref{nilpotent commute} (see \ref{ext preli}).

\begin{fact}\label{nilpotent}
     Let $G$ be a finite nilpotent group and $x,y \in G \setminus \{e\}$ be two elements such that $o(x)$ and $o(y)$ are co-prime to each other. Then $o(xy)=o(x)\cdot o(y)$. Moreover, if $M$ is any maximal cyclic subgroup of a finite non-cyclic nilpotent group $G$, then $p$ divides $|M|$ for all $p \in \pi(G)$.
\end{fact}

We now classify the non-trivial GE-classes (defined in \Cref{preli}) of a nilpotent group $G$ into two types based on their orders.

 \textit{Base class}: We call a GE-class $[x]$ with order $o(x)$ divisible by exactly one prime from $\pi(G)$ a \textit{base class} and its elements \emph{base elements}. A base element with the order as a positive power of a prime $p \in \pi(G)$ is called a $p$-base element. We denote the set of all base elements by $B$ and the set of all $p$-base elements by $B_p$ for a prime $p\in \pi(G)$.

 \textit{Non-base class}: We call a GE-class with order divisible by at least two primes from $\pi(G)$ a \textit{non-base class} and its elements \emph{non-base elements}. We denote the set of all non-base elements by $NB$.

In finite nilpotent groups, if $[x]$ and $[y]$ are two GE-classes of order $p^k$ and $q^l$, respectively, where $p$ and $q$ are distinct primes, then $[xy]$ is the GE-class of order $p^kq^l$ and $[xy]=gen(\langle xy \rangle)$ (using \Cref{nilpotent}). Moreover, by \Cref{adjacency of classes}, $[xy]$ is adjacent to both $[x]$ and $[y]$. By \Cref{gtf_CLT}, it can be easily observed that a non-base class of order $n$ is adjacent to exactly one base class of order $p^i$, $i\geq 1$ where $p$ is a prime and $p^i$ is a divisor of $n$.
 
\vspace{0.2cm}

The next lemma is similar to Theorem 2.6 \cite{doostabadi}.

\begin{lemma}\label{nil: path of length 4}
Let $G$ be a non-cyclic nilpotent group with $|G|=p^mq^n$, where $p$ and $q$ are distinct primes\footnote{Note that the condition on $|G|$ in \Cref{nil: path of length 4} is not necessary, but it is enough for our further discussion and makes the presentation simpler.} 
 and $m,n \geq 1$.  Let $u,v \in G \setminus \{e\}$ such that $\langle u \rangle \cap \langle v \rangle =\{e \}$ satisfying one of the following conditions: (i) Both $u$ and $v$ are $p$-base elements; (ii) Both $u$ and $v$ are $q$-base elements; (iii) Both $u$ and $v$ are non-base elements. Then, any $e$-avoiding shortest path between $u$ and $v$ in $Pow(G)$ is of length $4$.
\end{lemma}
\begin{proof}
    
    We first prove the following claim.
    \begin{claim}\label{e-avoiding path}
    Let $P$ be any $e$-avoiding path between $u$ and $v$ in $Pow(G)$. Then $P$ (including $u$ and $v$) must contain one $p$-base element and one $q$-base element. 
\end{claim}

\noindent\textit{Proof of \Cref{e-avoiding path}:} Let $a,b \in G$ be two adjacent elements in $Pow(G)$. If there exists a prime $p$ which divides both $o(a)$ and $o(b)$, then $\langle a \rangle \cap \langle b \rangle$ contains a $p$-order subgroup (because $\ang{a} \cap \ang{b}$ is a cyclic subgroup and \Cref{gtf_CLT} holds). Now, let $P: ug_1g_2\dots g_n v$ be an $e$-avoiding path between $u$ and $v$. For the sake of contradiction, assume that every vertex of $P$ has its order divisible by $p$. Then, $\langle u \rangle \cap \langle g_1 \rangle \cap \dots \cap \langle g_n \rangle \cap \langle v \rangle$ contains a $p$-order subgroup, which contradicts the fact that $\langle u \rangle \cap \langle v \rangle=\{e\}$. So, there is at least one $q$-base element in $P$. Similarly, we can say that $P$ contains at least one $p$-base element.\hfill $\lhd$


We now go over the conditions of \Cref{nil: path of length 4} one by one.

(i) Let $o(u)=p^{\alpha}$, $\al \geq 1$ and $o(v)=p^{\alpha '}$, $\al' \geq 1$. From \Cref{e-avoiding path}, any $e$-avoiding path between $u$ and $v$ in $Pow(G)$ contains at least one element $a$ of order $q^{\beta}$, where $\beta \geq 1$. Now by \Cref{remark power graph}, $a$ is not adjacent to either $u$ or $v$ in $Pow(G)$. Hence, any shortest $e$-avoiding path between $u$ and $a$ is of length at least $2$, and similarly, any shortest $e$-avoiding path between $a$ and $v$ is of length at least $2$.

(ii) The proof is similar to (i).

(iii) Let  $o(u)=p^{\alpha}q^{\beta}$, $\al,\be \geq 1$ and $o(v)=p^{\alpha '}q^{\beta'}$, $\al',\be' \geq 1$. From \Cref{e-avoiding path}, an $e$-avoiding path $P$ between $u$ and $v$ in $Pow(G)$, contains at least one element, say $a$, of order $p^{r}$, where $r \geq 1$ and one element, say $b$, of order $q^{r'}$, where $r'\geq 1$. Now, an $e$-avoiding path between $a$ and $b$ is of length at least $2$ (since $a$ and $b$ are not adjacent in $Pow(G)$ by \Cref{remark power graph}). 
So, the length of $P$ is at least $4$.
\end{proof}

The next lemma is crucial in cutting down the number of patterns for showing the lower bound of the oriented diameter of power graphs of nilpotent groups that are considered in \Cref{nil: OD strictly greater than 3}. 

\begin{lemma}\label{Fat Arrow Lemma}\textbf{(Uniformity lemma)}
    Let $G$ be a non-cyclic nilpotent group such that $|G|=p^mq^n$, where $p$ and $q$ are distinct primes and $m,n \geq 1$. Let $u,v \in G\setminus \{e\}$ such that $\langle u \rangle \cap \langle v \rangle = \{e\} $ satisfying one of the following conditions:
    (i) Both $u$ and $v$ are $p$-base elements ; 
    (ii) Both $u$ and $v$ are $q$-base elements ;
    (iii) Both $u$ and $v$ are non-base elements .
    In an orientation $\OO$ of $Pow(G)$ with diameter $3$, if $(u,e) \in \OO$ then $(v,e) \in \OO$. Also, if $(e,u) \in \OO$ then $(e,v) \in \OO$.
\end{lemma}
\begin{proof}
    Using \Cref{nil: path of length 4}, an undirected path between $u$ and $v$ of length at most $3$ in $Pow(G)$ must include identity $e$. So, $e$ must be included in any directed path between $u$ and $v$ using $\OO$. We prove by contradiction that if $(u,e) \in \OO$ then $(v,e) \in \OO$. Assume $(u,e) \in \OO$ but $(v,e) \notin \OO$. Now neither $(v,e)$ nor $ (e,u)$ are in $\OO$,  hence both directed paths from $v$ to $e$ and $e$ to $u$ are of length at least $2$, making the length of any directed path from $v$ to $u$ at least  $4$. This contradicts that $\OO$ is an orientation of diameter $3$. So, $(u,e) \in \OO$ implies $(v,e) \in \OO$. The reverse case is similar.
\end{proof}

\begin{lemma}\label{nil: od not 2}
    For a finite non-cyclic nilpotent group $G \notin \Gprime$,  the oriented diameter of $Pow(G)$ is at least $3$.
\end{lemma}

\begin{proof}
    A finite nilpotent group is a direct product of its Sylow subgroups. As $G$ is non-cyclic, at least one such Sylow subgroup, say Sylow $p$-subgroup $S_p$ is non-cyclic. 

    If $S_p$ is not also generalised quaternion, then using \Cref{unique p-subgroup of a p-group}, $S_p$ has at least two subgroups $P_1$ and $P_2$ of order $p$. Let $P_1=\ang{g_1}$ and $P_2=\ang{g_2}$. Now using \Cref{gtf_CLT}
    one can verify that $g_1$ and $g_2$ have no common neighbour other than $e$ in $Pow(G)$. Hence, $Pow(G)$ cannot have an orientation of diameter $2$.

    Now, let $S_p=S_2=Q_{2^n}$, $n \geq 3$. From the proof of (2) of \Cref{od of all p-groups}, we know that $S_2$ has at least three distinct cyclic subgroups $C_1,C_2,C_3$ of order $4$ such that $C_i \cap C_j =\{e,y\}$ where $1 \leq i < j \leq 3$ and $y$ is the unique element of order $2$ in $S_2$. Now, similar to the proof of (2) of \Cref{od of all p-groups}, we can argue that $e$ and $y$ are the only common neighbours of any vertex of $C_i$ and any vertex of $C_j$ in $Pow(G)$ and $OD(Pow(G))$ is at least $3$.
\end{proof}

We now characterise the nilpotent groups for which an oriented diameter of $3$ is not possible (see \cref{nil: OD strictly greater than 3}). For that, we use the following notations.

\textit{Subset Notations:}\label{subset construction}
Let $G$ be a non-cyclic nilpotent group and $|G|=2^mp^n$, where $p$ is an odd prime and $m,n \geq 1$. We use the notations $[x_1],[x_2],\dots,[x_r]$ to denote the GE-classes of order $2$ and the notations $[y_1],[y_2],\dots,[y_s]$ to denote the GE-classes of order $p$. We partition the set $B_2$ of $2$-base elements into sets $X_i$, $1 \leq i \leq r$ where  $X_i=\{u | \ u \in B_2 \text{ and } \  [x_i] \subseteq \UG\}$. Similarly, we partition the set $B_p$ of $p$-base elements into $Y_1,Y_2, \ldots, Y_s$. 
We partition the set of non-base elements of $G$ into $rs$ sets as follows: $A_{ij}=\{u | \ [x_i] \subseteq \UG \text{ and } [y_j] \subseteq \UG \}$ for $1 \leq i \leq r$ and $1 \leq j \leq s$. 

 The following fact can be verified using \Cref{gtf_CLT} and the fact that the intersection of two cyclic subgroups is also a cyclic subgroup. Moreover, it is used to prove \Cref{nil: OD strictly greater than 3}.
\begin{fact}\label{fact: about Aij}
    Let $i\neq i'$, $j \neq j'$. If $u\in A_{ij}, v \in A_{i'j'}$, then $\ang{u} \cap \ang{v}=\{e\}$.
\end{fact}

The following two statements, one an easy observation about power graphs and the other a lemma by Frobenius (1895) \cite{frobenius1895verallgemeinerung}, are used in the proof of \Cref{nil: OD strictly greater than 3}.

\begin{observation}
\label{pgtf_p}
     Let $G$ be a group and $x$ be an element of order $p$, where $p$ is any prime. Then, for any element $y \ (\neq e) \in G$ such that $\{x,y\}$ is an edge of $Pow(G)$, we have $x \in \ang{y}$. 
\end{observation}

\begin{lemma}\label{number of p-order subgroups in a group} \cite{frobenius1895verallgemeinerung} The number of $p$-order subgroups in a finite group $G$ is $k\cdot p +1$, where $k \geq 0$.
\end{lemma}

\begin{lemma}\label{nil: OD strictly greater than 3}
Let $G \notin \Gprime$  be a non-cyclic nilpotent group. If $G$ satisfies all of the following conditions: (a) $|G|=2^mp^n$, where $p$ is an odd prime, $m,n \geq 1$; (b) $G$ has a maximal cyclic subgroup of order $2p^{\beta}$, for some $1 \leq \beta \leq n$; (c) $G$ has at least two subgroups of order $p$; (d) $G$ has at least two subgroups of order $2$, then the oriented diameter of $Pow(G)$ is $4$.
\end{lemma}

\begin{proof} The proof is by contradiction and is divided into two main steps: In \textit{Step 1}, we show that if $Pow(G)$ has an orientation with diameter 3, then it must follow one of the 8 general patterns that are discussed below. In \textit{Step 2}, we show that each of these $8$ patterns gives a contradiction. As we will see, some of these patterns are just symmetric versions of other patterns.

\vspace{0.2cm}
    
   \noindent \textbf{Step 1:} We first note that all maximal cyclic subgroups of $G$ are of order $2^{m'}p^{n'}$, $m',n'\geq 1$ (by \Cref{nilpotent}). Therefore, by \Cref{4 upper bound}, $Pow(G)$ has an orientation with diameter 4. Moreover, by \Cref{nil: od not 2}, $OD(Pow(G))\geq 3$. 
    We show that $Pow(G)$ cannot have oriented diameter $3$. For the sake of contradiction, let $\OO$ be an orientation of $Pow(G)$ with diameter 3.  We use the same notations $A_{ij}$, $i\in \{1,2,\dots,r\}$ and $j\in \{1,2,\dots,s\}$ as defined above. Since (c) and (d) hold, we note that $r,s\geq 3$ by \Cref{number of p-order subgroups in a group}. We start by picking an element $v\in A_{11}$. Without loss of generality, let $(v,e)\in \OO$. We show that for any $u\in NB$, we must have $(u,e)\in \OO$. 
    
    Let $u\in A_{ij}$ where $1<i\leq r$ and $1<j\leq s$. By \Cref{fact: about Aij}, $\ang{u} \cap \ang{v}= \{e\}$ and hence $(u,e)\in \OO$ using \Cref{Fat Arrow Lemma}. As $u$ is arbitrarily chosen, we have $(u,e)\in \OO$, for all $u\in A_{ij}$ where $1<i\leq r$ and $1<j\leq s$. Now if $u \in A_{1j}$ (or $A_{i1}$), we have a set $A_{kk}$ where $k\notin \{1,j\}$ (or $k \notin \{i,1\}$), as $r,s\geq 3$.  Note that for  $w \in A_{kk}$, $\ang{v}\cap\ang{w}=\{e\}$ and $\ang{u} \cap\ang{w}=\{e\}$ (by \Cref{fact: about Aij}). So applying \Cref{Fat Arrow Lemma} to $v$ and $w$ gives  $(w,e) \in \OO$. Another application of \Cref{Fat Arrow Lemma} on $u$ and $w$ shows that $(u,e)\in \OO$. As $u$ is arbitrarily chosen, we have $(u,e)\in \OO$ for all $u\in A_{1j}$ (or $A_{i1}$) where $i \in \{1,2,\dots,r\}$, $j \in \{1,2,\dots,s\}$. The case $(e,v)\in \OO$ similarly implies that $(e,u)\in \OO$ for all $u\in NB$.  

Hence, we have either $(u,e) \in \OO$ for all $u \in NB$, or $(e,u) \in \OO$ for all $u \in NB$. We denote these by $NB \xrightarrow{}\{e\}$ and $\{e\} \xrightarrow{} NB$, respectively. In general for an orientation $\OO$ and two sets $A$ and $B$, we write $A \xrightarrow{} B$, if $(a,b)\in \OO$ for all $a\in A,b \in B$.

    In a similar way, using \Cref{Fat Arrow Lemma}, we can show that either $(u,e) \in \OO$ for all $u$ in the set $B_2$ of all 2-base elements (we use the shorthand $B_2 \xrightarrow{}\{e\}$ to denote this case) or  $(e,u) \in \OO$ for all $u \in B_2$ (denoted by $\{e\}\xrightarrow{}$ $B_2$). We can also show that either $(u,e) \in \OO$ for all $u$ in the set $B_p$ of all $p$-base elements (denoted by $B_p \xrightarrow{}\{e\}$) or  $(e,u) \in \OO$ for all $u \in B_p$ (denoted by $\{e\}\xrightarrow{}$ $B_p$).

    The above discussion shows that there are $8$ possible patterns in $\OO$. 

    \vspace{0.2cm}
    
    \noindent \textbf{Step 2:} Now we will inspect all the patterns one by one.

\textit{Pattern 1}: $NB\ra \{e\}, \ B_2\ra \{e\}, \ B_p\ra \{e\}$. This pattern does not yield a strong orientation since there is no outward edge from $e$.

  \textit{Pattern 2:} $\{e\}\ra NB, \ B_2\ra \{e\}, \ B_p \ra \{e\}$.
    In this pattern, any directed path containing $e$ from any non-base element to any base element is of length at least $4$. Now, we show that there exists at least a pair of vertices $a$ and $b$ such that we can not have a directed $e$-avoiding path from $a$ to $b$ of length at most $3$.  By condition $(b)$, $G$ has a maximal cyclic subgroup $C$ of order $2p^{\beta}$, for some $1 \leq \beta \leq  n$. Now, by condition $(c)$, $G$ has at least two subgroups of order $p$. Hence by \Cref{gtf_CLT}, $G$ has a subgroup $\ang{v}$ of order $p$ such that $C \cap \ang{v}=\{e\}$. We need the following claim.
    
    \begin{claim}
    \label{mandatory_2}
        Let $C$ be a maximal cyclic subgroup of $G$ of order $2p^{\beta}$, $\beta \geq 1$. Let $u$ be a non-base element in $C$ and $v$ be an element of order $p$ not in $C$. Then, there is no $e$-avoiding path between $u$ and $v$ of length at most $2$ in $Pow(G)$. Moreover, if $P: uw_1w_2 v$ is an $e$-avoiding path of length $3$ between $u$ and $v$ in $Pow(G)$, then $w_1$ has to be the unique element of order $2$ in $C$.
    \end{claim}
   \noindent \textit{Proof of \Cref{mandatory_2}:}
         Let $x$ and $y$ be elements of order $2$ and $p$ respectively in $C$. Then by \Cref{gtf_CLT}, $\ang{y} \leq \ang{u} \leq C$. First, we show that there is no $e$-avoiding path between $u$ and $v$ of length at most $2$ in $Pow(G)$. Since $u$ does not generate $v$, $\{u,v\}$ is not an edge (by \Cref{pgtf_p}). Now, if possible, let $w\neq e$ be a common neighbour of $u$ and $v$. By \Cref{pgtf_p}, 
         $\ang{v} \leq \ang{w}$. Moreover since $C \cap \ang{v}=\{e\}$, we must have $u \in \ang{w}$ (otherwise, 
         $v $ belongs to $\ang{u}$, a contradiction to our assumption) and hence $\ang{u} \leq \ang{w}$. This implies that $\ang{w}$ contains $\ang{y}$ and also $\ang{v}$,  a contradiction to \Cref{gtf_CLT}. Hence, no $e$-avoiding path between $u$ and $v$ in $Pow(G)$ of length at most $2$ exists. 

    Now, let $P:uw_1w_2v$ be an $e$-avoiding path as mentioned in the statement of the claim. At first, observe that if $o(w_1)=2^{\al}$ for some $\alpha\geq 1$, then\footnote{This is because in $Pow(G)$ if an element $x$ of prime power order is adjacent to an element $y$ of non-prime power order, then $x \in \ang{y}$.} $w_1 \in \ang{u} \subseteq C$. Now, as $o(w_1)$ divides the order of the subgroup $C$, we have $o(w_1)=2$ and hence $w_1=x$. So, to prove the claim, it is enough to show that $p$ does not divide $o(w_1)$. For the sake of contradiction, let $p| o(w_1)$. 
        
        If $w_1 \in \ang{u}$, then $\ang{w_1}$ is a subgroup 
        $C$. Since $p| o(w_1)$, $\ang{w_1}$ contains a unique subgroup of order $p$ (by \Cref{gtf_CLT}) and hence $\ang{y} \leq \ang{w_1}$. If $u \in \ang{w_1}$, then $\ang{y} \leq \ang{u} \leq \ang{w_1}$. Hence $\ang{y} \leq \ang{w_1}$ in both the cases.

        Now $w_2$ is a common neighbour of $w_1$ and $v$. As $\{w_2, v\} \in Pow(G)$ and $o(v)=p$, we have $v \in \ang{w_2}$, by \Cref{pgtf_p}. On the other hand, $\{w_1,w_2\}$ implies either $w_1 \in \ang{w_2}$ or $w_2 \in \ang{w_1}$. If $w_1 \in \ang{w_2}$, then $\ang{w_2}$ contains $\ang{w_1}$. Hence $\ang{w_2}$ contains two distinct subgroups $\ang{y}$ and $\ang{v}$ of order $p$, a contradiction to \Cref{gtf_CLT}. Also, if $w_2 \in \ang{w_1}$, then 
        $\ang{v} \leq \ang{w_1}$ (since $v \in \ang{w_2}$). Again, this contradicts \Cref{gtf_CLT}. So, $p$ does not divide $o(w_1)$. 
        \hfill $\lhd$

    \vspace{0.2cm}
        Let $C_1$ be a maximal cyclic subgroup of order $2p^{\beta}$ containing $x$ as the unique $2$ order element and $y_1$ as a $p$ order element. Let $u_1$ be a non-base element of $C_1$. Let $v_2$ be an element of order $p$ outside $C_1$. Using \Cref{mandatory_2}, any directed $e$-avoiding path of length $3$ from $u_1$ to $v_2$ must use the edge $(u_1,x)$. So, the path should be of the form: $u_1xgv_2$, where $g$ is some element\footnote{Note that $g$ is a non-base element.} such that $\ang{g}$ contains both $x$ and $v_2$ (by \Cref{pgtf_p}). Hence, to have a directed $e$-avoiding path from $u_1$ to $v_2$ of length $3$, we must have $(x,g) \in \OO$. 
        Now if $C_2$ is a maximal cyclic subgroup containing $g$ (and hence containing $x$), then by \Cref{obs mcs} (see \Cref{nn}), the order of $C_2$ is $2p^{\beta'}$, for some $\beta' \geq 1$. To have a directed $e$-avoiding path from $g$ to $y_1$ of length at most $3$, we must put $(g,x) \in \OO$, by using \Cref{mandatory_2}. This contradicts our previous requirement of $(x,g) \in  \OO$. Hence, Pattern 2 is not possible in $\OO$.

    \textit{Pattern 3}: $NB\ra \{e\},\  \{e\} \ra B_2,\ B_p\ra \{e\}$. Let $y$ and $y'$ be two $p$-base elements such that $[y]\neq [y']$, i.e., $\ang{y} \cap \ang{y'} = \{e\}$. Using \Cref{nil: path of length 4}, a directed path $\overrightarrow{P}$ from $y$ to $y'$ of length $3$ must pass through $e$. Moreover, since $B_p\ra \{e\}$, $\overrightarrow{P}$ must have $(y,e)$.
    Note that the only outward edges from $e$ are towards the $2$-base elements. Also, since $y'$ is a $p$-base element, any directed path from a $2$-base element to $y'$ is of at least length $2$, by \Cref{remark power graph}. 
    Therefore, the length of a directed path from $y$ to $y'$ is at least $4$. Hence, we cannot have an orientation with diameter $3$.

    \textit{Pattern 4}: $NB\ra \{e\}, B_2 \ra \{e\}, \{e\}\ra B_p$. As done in Pattern 3, we can similarly argue that there is no directed path of length at most $3$ from $x$ to $x'$, where $x$ and $x'$ are $2$-base elements and $[x]\neq[x']$.

    The last four patterns are symmetric to the first four patterns and can be dealt with using the following simple observation.

\begin{observation}\label{symmetric choice of orientation}
    
Let $\mathfrak{X} =(V,E)$ be a directed graph. Also, let $\mathfrak{X}_{rev}=(V, E_{rev})$ be the directed graph where $E_{rev}$ is the set of edges obtained by reversing the directions of all the edges in $E$. Then $diam(\mathfrak{X})=diam(\mathfrak{X}_{rev})$.

Let $\OO$ be a partial orientation of an undirected graph $X=(V,E)$, $X_\OO$ be the directed graph $(V,\OO)$ and $A \subseteq \OO$. If $diam(X_{\OO})=d$, then there exists a partial orientation $\OO'$ of $X$ containing $A_{rev}$ such that $diam(X_{\OO'})=d$.\hfill $\lhd$

\end{observation}

     \textit{Pattern 5}: $\{e\}\ra NB, \ \{e\}\ra B_2, \ \{e\} \ra B_p$. By \Cref{symmetric choice of orientation}, this is symmetric to Pattern 1. By `symmetric', we mean that getting a partial orientation containing Pattern 5 with diameter $3$ would imply that there is a partial orientation containing Pattern 1 with diameter $3$.
    
    \textit{Pattern 6}: $\{e\}\ra NB, \ B_2\ra \{e\}, \ \{e\}\ra B_p$. By \Cref{symmetric choice of orientation}, this is symmetric to Pattern 3.

    \textit{Pattern 7:} $\{e\}\ra NB, \ \{e\}\ra B_2, \ B_p \ra \{e\}$. By \Cref{symmetric choice of orientation}, this is symmetric to Pattern 4.

    \textit{Pattern 8:} $NB \ra \{e\},\ \{e\} \ra B_2, \ \{e\} \ra B_p$. By \Cref{symmetric choice of orientation}, this is symmetric to Pattern 2.

    So we have shown that none of the $8$ patterns is satisfied in an orientation of $Pow(G)$ with diameter $3$. Hence, it is proved that if $G$ satisfies the given conditions (a)-(d), then we cannot orient $Pow(G)$ with diameter $3$.
\end{proof}

We now state the main result on the oriented diameter of power graphs of finite non-cyclic nilpotent groups which are not in $\Gprime$.


\begin{theorem}\label{nil: main result}
    Let $G \notin \Gprime$ be a finite non-cyclic nilpotent group. Then the oriented diameter of $Pow(G)$ is $3$ if and only if $G$ satisfies at least one of the following conditions: (a) $|G|$ has at least two distinct odd prime factors; (b) $G$ has no maximal cyclic subgroup of order $2p^{\beta}$, $1\leq \beta \leq n$, where $p$ is an odd prime; (c) $G$ has unique $p$-order subgroup, where $p$ is an odd prime; (d) $G$ has unique $2$-order subgroup. Otherwise, the oriented diameter of $Pow(G)$ is $4$.
\end{theorem}

Examples of groups for each of the conditions are in  \Cref{nil: example}. We first state the following four lemmas: \Cref{nil: two odd primes}, \Cref{nil: no mcs of order 2pk}, \Cref{nil: one p-order} and \Cref{nil: one 2-order}, which are used to prove the above theorem.

\begin{lemma}\label{nil: two odd primes}
    Let $G$ be a non-cyclic nilpotent group. If $|G|$ is divisible by at least two distinct odd primes, then the oriented diameter of $Pow(G)$ is $3$.
\end{lemma}

\begin{lemma} \label{nil: no mcs of order 2pk}
    Let $G$ be a non-cyclic nilpotent group of order $2^mp^n$, where $p$ is an odd prime and $m,n \geq 1$.  If $G$ has no maximal cyclic subgroup of order $2p^{\beta}$, $1\leq \beta \leq n$, then the oriented diameter of $Pow(G)$ is $3$.
\end{lemma}

\begin{lemma}\label{nil: one p-order}
    Let $G$ be a non-cyclic nilpotent group and $|G|=2^mp^n$, where $p$ is an odd prime and $m,n \geq 1$. If $G$ has unique subgroup of order $p$, then the oriented diameter of $Pow(G)$ is $3$.
\end{lemma}

\begin{lemma} \label{nil: one 2-order}
    Let $G$ be a non-cyclic nilpotent group and $|G|=2^mp^n$, $m, n \geq 1$, where $p$ is an odd prime. If $G$ has a unique subgroup of order $2$, then the oriented diameter of $Pow(G)$ is $3$.
\end{lemma}

\noindent \textbf{Proof of \Cref{nil: main result}:} If $G$ does not satisfy any of the conditions (a)-(d), then by \Cref{nil: OD strictly greater than 3}, $OD(Pow(G))=4$. Now consider the opposite direction. The case when $G$ satisfies condition (a) is handled in \Cref{nil: two odd primes}. Now note that for the remaining cases, i.e., when $G$ satisfies condition (b) or (c) or (d), it is enough to consider $|G|=2^mp^n$, where $p$ is an odd prime and $m,n \geq 1$. Hence, by applying \Cref{nil: no mcs of order 2pk}, \Cref{nil: one p-order} and \Cref{nil: one 2-order}, the oriented diameter of $Pow(G)$ is $3$ when $G$ satisfies condition (b), (c) and (d) respectively.
\hfill $\square$

\vspace{0.2cm}

We now prove \Cref{nil: two odd primes} and \Cref{nil: one 2-order} in the rest of this section. The proof techniques of \Cref{nil: no mcs of order 2pk} and \Cref{nil: one p-order} are almost similar to \Cref{nil: two odd primes}, and hence, we have put the proofs of these two lemmas in \Cref{app_nil_no mcs of order 2pk} and \Cref{app_nil_one p-order} respectively. Now due to \Cref{obs od is at most diam of partial orient} and \Cref{nil: od not 2}, in order to prove each of the four lemmas, it is enough to give a partial orientation of diameter 3. The partial orientations used in \Cref{nil: two odd primes}, \Cref{nil: no mcs of order 2pk} and \Cref{nil: one p-order} are different but involve some common partial orientations, namely $\OO_1,\OO_2$ and $\OO_3$. These common partial orientations are described in \Cref{Remark_2} below. In \Cref{lem: pathtwo}, we prove that $\OO_2$ and $\OO_3$ themselves can establish directed paths of length at most $2$ between certain sets of vertices. In each of the three lemmas: \Cref{nil: two odd primes}, \Cref{nil: no mcs of order 2pk} and \cref{nil: one p-order}, we augment $\cup_{i=1}^3 \OO_i$ with suitable partial orientations. Whereas, in \Cref{nil: one 2-order} we design a completely different partial orientation.

\begin{construction}\label{Remark_2}
Let $G \notin \Gprime$ be a finite non-cyclic nilpotent group. The descriptions of partial orientations $\OO_1,\OO_2,\OO_3$ of $Pow(G)$ are as follows:    
\end{construction}
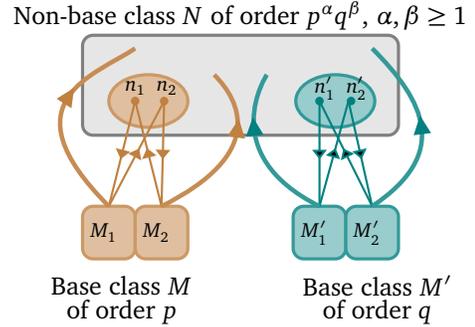
\begin{wrapfigure}{r}{0.45\textwidth}

    \centering
        \begin{tikzpicture}[scale=0.35]
        
        \coordinate (a) at (2,7);
        \coordinate (b) at (3,11);
        \coordinate (c) at (4,7);
        \coordinate (d) at (4.1,11);
        \coordinate (e) at (10,7);
        \coordinate (f) at (10,11);
        \coordinate (g) at (12,7);
        \coordinate (h) at (11.1,11);

        \filldraw[color=black!50, fill=black!10, very thick][rounded corners] (1,9.7) rectangle (13,13.5);
         \filldraw[color=brown!80, fill=brown!50, very thick][rounded corners] (1,5) rectangle (3,7);
        \filldraw[color=brown!80, fill=brown!50, very thick][rounded corners] (3,5) rectangle (5,7);
        \filldraw[color=teal!80, fill=teal!40, very thick][rounded corners] (9,5) rectangle (11,7);
        \filldraw[color=teal!80, fill=teal!40, very thick][rounded corners] (11,5) rectangle (13,7);

         \filldraw[color=brown!80, fill=brown!50, very thick][] (3.5,11) ellipse (1.5cm and 1cm);
         \filldraw[color=teal!80, fill=teal!40, very thick][] (10.6,11) ellipse (1.5cm and 1cm);

        \filldraw [brown] (3,11) circle(4pt);
        \filldraw [brown] (4.1,11) circle(4pt);
        \filldraw [teal] (10,11) circle(4pt); 
        \filldraw [teal] (11.2,11) circle(4pt); 

        \draw (3,11.5) node[][scale=0.8] {$n_1$};
        \draw (4.2,11.5) node[][scale=0.8] {$n_2$};
         \draw (10.2,11.5) node[][scale=0.8] {$n_1'$};
        \draw (11.4,11.5) node[][scale=0.8] {$n_2'$};

        \draw (1,6) node[right][scale=0.8] {$M_1$};
        \draw (3,6) node[right][scale=0.8] {$M_2$};
        \draw (9,6) node[right][scale=0.8] {$M_1'$};
        \draw (11,6) node[right][scale=0.8] {$M_2'$};
        
        \draw[brown, thick] (a) -- (b) node[ currarrow, pos=0.5, xscale=0.6, sloped, scale=-1, color=brown] {} ;
        \draw[brown, thick] (b) -- (c) node[ currarrow, pos=0.5, xscale=0.6, sloped, scale=-1, color=brown] {} ;
        \draw[brown, thick] (c) -- (d) node[ currarrow, pos=0.5, xscale=0.6, sloped, scale=-1, color=brown] {} ;
        \draw[brown, thick] (d) -- (a) node[ currarrow, pos=0.5, xscale=-0.6, sloped, scale=-1, color=brown] {} ;

       \draw[teal,thick] (e) -- (f) node[ currarrow, pos=0.5, xscale=-0.6, sloped, scale=1] {} ;
       \draw[teal,thick] (f) -- (g) node[ currarrow, pos=0.5, xscale=-0.6, sloped, scale=1] {} ; 
       \draw[teal,thick] (g) -- (h) node[ currarrow, pos=0.5, xscale=0.6, sloped, scale=1] {} ; 
       \draw[teal,thick] (h) -- (e) node[ currarrow, pos=0.5, xscale=0.6, sloped, scale=1] {} ;

        \draw[ line width=0.6mm,color=brown!90] (2,7) .. controls (-2,11) and (1.5,12) .. (3,13) node[ currarrow, pos=0.5, xscale=1, sloped, scale=1, color=brown] {};
        \draw[ line width=0.6mm,color=brown!90] (4,7) .. controls (8,9) and (7,11) .. (6,12) node[ currarrow, pos=0.5, xscale=1, sloped, scale=1, color=brown] {};
        
         \draw[ line width=0.6mm,color=teal!80] (10,7) .. controls (6,9) and (8,12) .. (8,12) node[ currarrow, pos=0.45, xscale=-1, sloped, scale=1, color=teal] {};
        \draw[ line width=0.6mm,color=teal!80] (12,7) .. controls (15,9) and (14,11) .. (12,13) node[ currarrow, pos=0.5, xscale=-1, sloped, scale=1, color=teal] {};

        \draw (7,14.25) node[][scale=1] {Non-base class $N$ of order $p^{\al}q^{\be}$, $\al,\be \geq 1$};
        \draw (2.5,4) node[][scale=1] {Base class $M$};
        \draw (2.5,3) node[][scale=1] {of order $p$};
        \draw (12.2,4) node[][scale=1] {Base class $M'$}; 
        \draw (12.2,3) node[][scale=1] {of order $q$};

         \end{tikzpicture}
        \caption{Illustration of $\OO_2$. The directed edges of $E(\{n_1,n_2\}, M_1 \cup M_2)$ and $E(\{n'_1,n'_2\}, M'_1 \cup M'_2)$ form two $C_4$-gadgets.}
         \label{fig:gmgn}
\end{wrapfigure}

 $\OO_1$ : From all $u \in NB$, we orient the edges towards $e$ as $(u,e)$. Also, from $e$, we orient the edges towards all $u \in B$ as $(e,u)$.

$\OO_2$ :  At first, we arbitrarily partition each base class $M$ of odd prime order into two non-empty subsets $M_1$ and $M_2$. Let $N$ be a non-base class and $p$ be an odd prime divisor of the order of $N$. Due to \Cref{gtf_CLT}, $N$ is adjacent to exactly one base class $M$ of order $p$.  
At first, we mark two elements of $N$ as $n_{1},n_{2}$ (choices of $n_1$ and $n_2$ depend on $M$, as discussed in the note below). 
        Now, we make directed $4$-cycles in the edges of $E(N,M)$ (see \Cref{fig:gmgn}) as follows:  For all $u_1 \in M_1$ and $u_2 \in M_2$ we put $(n_1,u_1)$, $(u_1,n_2)$, $(n_2,u_2)$, $(u_2,n_1)$ in $\OO_2$. We call the directed subgraph formed by the directed edges of $E(\{n_1,n_2\},M_1\cup M_2)$ due to $\OO_2$ - `$C_4$-\textit{gadget}'. This naming is due to the presence of several directed $C_4$ in $E(\{n_1,n_2\},M_1\cup M_2)$ after $\OO_2$. Moreover we call $n_1,n_2$ the \textit{gadget anchor points} in $N$ for $M$.
        Now, the edges in $E(M,N)$ of the form $\{u,v\}$, where $u\in M$ and $v \in N\setminus\{n_1,n_2\}$ are  oriented as $(u,v)$.

        \textit{Note:}  While introducing  $\OO_2$, we choose a disjoint pair of gadget anchor points in $N$ for each base class  $M$ of odd prime order adjacent to $N$. This is possible as the number $k$ of base classes of odd prime order adjacent to $N$ equals the number of odd prime divisors of $r$, where $r$ is the order of $N$ and also noting that $|N|=\phi(r)>2^k\geq 2k$.

$\OO_3$: From any base element of order $p^{\al}$, $\al \geq 2$, where $p$ is an odd prime in $\pi(G)$, we orient the edges towards all the adjacent non-base elements.
\hfill $\lhd$

\begin{lemma}
\label{lem: pathtwo}
    Let $G$ be a finite non-cyclic nilpotent group such that $|G|$ has two odd prime factors $p$ and $q$. Then using partial orientations $\OO_2$ and $\OO_3$ as stated in \Cref{Remark_2}, we have the following:
    
    (1) There is a directed path of length $2$ between any vertex of a base class of order $p$ and any vertex of a base class of order $q$ using $\OO_2$, where $p$ and $q$ are distinct.

    (2) There is a directed path of length at most $2$ from any vertex of a base class of order $p^{\al}$, $\al \geq 2$ to any vertex of a base class of order $q$, using $\OO_3$ and $\OO_2$, where $p$ and $q$ may or may not be distinct.

\end{lemma}

\begin{proof}
    We only prove (1) as (2) is similar.  Let $M$ and $M'$ be two base classes of order $p$ and $q$, respectively. Then, we need to show that between any pair of vertices $m \in M$ and $m' \in M'$, there is a directed path of length $2$. 
    As $G$ is a nilpotent group, $mm'$ is a non-base element (by \Cref{nilpotent}). Let $N$ be the non-base class $[mm']$. Since $mm'$ is adjacent to both $m$ and $m'$ in $Pow(G)$, $N$ is adjacent to both the base classes $M$ and $M'$ (by using \Cref{equivalence classes form a clique}). Hence, we have oriented the edges in $E(N,M)$ and $E(N,M')$ as described in $\OO_2$. 

    Let $n_1,n_2$ be the gadget anchor points in $N$ involved in the $C_4$-gadget with $M$. Also, let $n_3, n_4$ be the gadget anchor points in $N$ involved in the gadget with $M'$. From our discussion in \Cref{Remark_2}, the vertices $n_1,n_2,n_3,n_4$ are distinct. Now, for any $m' \in M'$, we have either $(n_3,m')$ or $(n_4,m')$. Also, as $n_3$ and $n_4$ are not involved in the $C_4$-gadget with $M$, we have $(m,n_3)$ and $(m,n_4)$ for all $m \in M$. Hence, we have a directed path of length $2$ from any $m \in M$ to any $m' \in M'$ via $n_3$ or $n_4$. \end{proof}

\begin{fact}\label{remark common element}
    Let $G$ be a finite nilpotent group such that the set of prime divisors $\pi(G)$ contains $2$ and at least two distinct odd primes $p$ and $q$. Then using \Cref{nilpotent}, any element of order $2^{\al}, \ \al \geq 1$ and any element $p^{\be}, \ \be \geq 1$ have one common neighbour of order $2^{\al}p^{\be}q$ in $Pow(G)$.
\end{fact}

  \noindent \textbf{Proof of \Cref{nil: two odd primes}:} Due to \Cref{nil: od not 2} and \Cref{obs od is at most diam of partial orient}, it is sufficient to give a partial orientation of $Pow(G)$ with diameter $3$. For that purpose, if $|G|$ is even, then along with partial orientations $\OO_1,\OO_2,\OO_3$ as stated in \Cref{Remark_2}, we use the following partial orientations (see \Cref{fig:lemma 23}). If $|G|$ is odd, we will see below that the partial orientations $\OO_1,\OO_2 , \OO_3$ are sufficient.

    $\OO_4$: From any base element of order $2$, we orient the edges towards all the adjacent non-base elements of order $2^{\alpha}p^{\beta}$, where $p$ is any odd prime in $\pi(G)$ and $\alpha, \beta \geq 1$.
    
      $\OO_5$: Let $N$ be a non-base class of order $2^{\alpha}t$, where $\al\geq 2$ and $t \ (\neq 1)$ is co-prime to $2$. Also, let $M$ be the unique base class of order $2^2$ that is adjacent to $N$. We orient the edges in $E(N,M)$ similarly to $\OO_2$ as stated in \Cref{Remark_2}. Here also, the choices of gadget anchor points of $N$ for $M$ depend on $M$ as in \Cref{Remark_2}. In other words,  while introducing $\OO_5$ in a non-base class $N$ of order $2^{\alpha}t$, where $\al\geq 2$ and $2 \nmid t$, we select a pair of gadget anchor points $\{n_1,n_2\}$ in $N$ for the base class of order $2^2$ adjacent to $N$ in such a way that neither $n_1$ nor $n_2$ has been used as gadget anchor point in $N$ while introducing $\OO_2$. This is possible as the number $k$ of base classes of odd prime order adjacent to $N$ equals the number of prime divisors of $t$, and $|N|=\phi(2^{\al}t)>2^{(k+1)}\geq 2(k+1)$.

      $\OO_6$: From any base element of order $2^{\al}, \al \geq 3$, we orient the edges towards all the adjacent non-base elements. Note that these non-base elements are of order $2^{\delta}t$, where $\delta \geq \al$ and $t \ (\neq 1)$ is co-prime to $2$.
   
       $\OO_7$: From any non-base element of order $2p^{\al}q^{\be}$, $\al,\be \geq 1 $, where $p$ and $q$ are any two distinct odd primes in $\pi(G)$, we orient the edge towards the adjacent unique base element of order $2$.
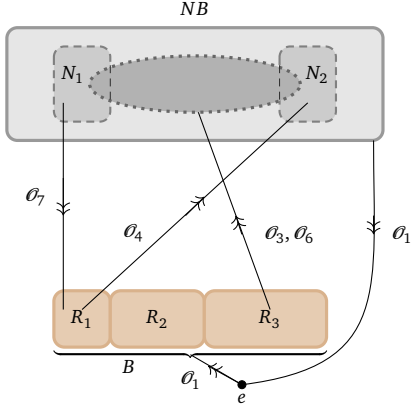
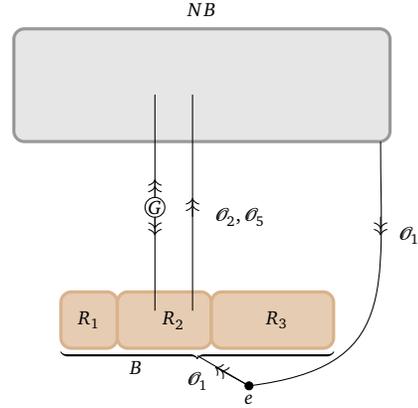
\begin{figure}[hpt!]
\begin{subfigure}[b]{0.47\textwidth}
\centering
        \begin{tikzpicture}[scale=0.25]
        
        \coordinate (a) at (2,7);
        \coordinate (b) at (3,11);
        \coordinate (c) at (4,7);
        \coordinate (d) at (4.1,11);
        \coordinate (e) at (10,7);
        \coordinate (f) at (10,11);
        \coordinate (g) at (12,7);
        \coordinate (h) at (11.1,11);

        \filldraw[color=black!40, fill=black!10, very thick][rounded corners] (2.5,20) rectangle (22.5,14);

         \filldraw[color=brown!60, fill=brown!40, very thick][rounded corners] (5,6) rectangle (8,3);
        \filldraw[color=brown!60, fill=brown!40, very thick][rounded corners] (8,6) rectangle (13,3);
        \filldraw[color=brown!60, fill=brown!40, very thick][rounded corners] (13,6) rectangle (19.5,3);

        \filldraw[color=black!50, fill=black!20, thick, densely dashed][rounded corners] (5,19) rectangle (8,15);
        \filldraw[color=black!50, fill=black!20, thick, densely dashed][rounded corners] (17,19) rectangle (20,15);

        \path [draw] (15,1) -- (12.26,2.6);
        \path [draw] (5.5,5) -- (5.5,16);
        \path [draw] (6.5,5) -- (18.5,16);
        \path [draw] (16.5,5) -- (12.5,16);

        \path [draw] (15,1) .. controls (23,2) and (22,6) ..  (22,14);
        \path [draw,<<-][] (22,9) -- (22,10);
        
         \path [draw,->>][] (5.5,12) -- (5.5,10);
        \path [draw,->>][] (12,10) -- (13,11);
        \path [draw,->>][] (15,9) -- (14.67,10);

   
                   \filldraw [black] (15,1) circle(6pt);

         \filldraw[color=black!60, fill=black!30, very thick, dotted] (12.5,17) ellipse (5.6cm and 1.5cm);
         
          \draw [color=black!80, densely dashed] (8,18) -- (8,16); 
          \draw [color=black!80, densely dashed] (17,18) -- (17,16);

  \draw [decorate,
    decoration = {calligraphic brace}][thick] (19.5,2.8) --  (5,2.8);

        \draw (12.5,21) node[][scale=0.7] {$NB$};
        \draw (9,2.7) node[below][scale=0.7]{$B$};
        \draw (6.5,4.5) node[][scale=0.7]{$R_1$};
        \draw (10.5,4.5) node[][scale=0.7]{$R_2$};
        \draw (16.5,4.5) node[][scale=0.7]{$R_3$};
        \draw (6,17.5) node[][scale=0.7]{$N_1$};
        \draw (19,17.5) node[][scale=0.7]{$N_2$};
        \draw (15,0.3) node[][scale=0.7]{$e$};
        \draw (4,11) node[][scale=0.7]{$\OO_7$};
        \draw (9.2,9.1) node[][scale=0.7]{$\OO_4$};
        \draw (17.5,9) node[][scale=0.7]{$\OO_3, \OO_6$};
        \draw (23.5,9) node[][scale=0.7]{$\OO_1$};
        
        \draw (12.25,1.25) node[][scale=0.7]{$\OO_1$};
         \path [draw,->>][] (15,1) -- (13.2,2);
              
         \end{tikzpicture}
        \caption{Connection between elements of $R_1$ and their adjacent elements in $N_1$ and $N_2$, where  $N_1=\{v \in NB | o(v)=2p^{\al}q^{\be},\ \al,\be \geq 1, \text{ where $p$ and $q$ are any odd primes in $\pi(G)$} \}$ and $N_2=\{v \in NB | o(v)=2^{\al}p^{\be},\ \al, \be \geq 1, \text{ where $p$ is an odd prime in $\pi(G)$}\}$.\\ Connection between elements of $R_3$ and their adjacent elements in $NB$. }
         \label{fig:lemma 23_a}
\end{subfigure} \hspace{1 cm}
\begin{subfigure}[b]{0.47\textwidth}
\centering
   \begin{tikzpicture}[scale=0.25]

        \filldraw[color=black!40, fill=black!10, very thick][rounded corners] (2.5,20) rectangle (22.5,14);

        \filldraw[color=brown!60, fill=brown!40, very thick][rounded corners] (5,6) rectangle (8,3);
        \filldraw[color=brown!60, fill=brown!40, very thick][rounded corners] (8,6) rectangle (13,3);
        \filldraw[color=brown!60, fill=brown!40, very thick][rounded corners] (13,6) rectangle (19.5,3);


        \path [draw] (15,1) -- (12.26,2.6);
%
%
%
%

\path [draw] (10,5) -- (10,16.5);
            \path [draw] (12,5) -- (12,16.5); 
        \path [draw,->>][] (12,10) -- (12,11);
                \path [draw] (15,1) .. controls (23,2) and (22,6) ..  (22,14);
        \path [draw,<<-][] (22,9) -- (22,10);
         \filldraw [white] (10,10.5) circle(15pt);
         \draw [] (10,10.5) circle(15pt);  
         \draw (10,10.5) node[][scale=0.7] {$G$};
         \path [draw,->>][] (10,11) -- (10,12);
        \path [draw,->>][] (10,10) -- (10,9);
                  \filldraw [black] (15,1) circle(6pt);

  \draw [decorate,
    decoration = {calligraphic brace}][thick] (19.5,2.8) --  (5,2.8);

        \draw (12.5,21) node[][scale=0.7] {$NB$};
        \draw (9,2.7) node[below][scale=0.7]{$B$};
        \draw (6.5,4.5) node[][scale=0.7]{$R_1$};
        \draw (11,4.5) node[][scale=0.7]{$R_2$};
        \draw (16.5,4.5) node[][scale=0.7]{$R_3$};
        \draw (15,0.3) node[][scale=0.7]{$e$};
        \draw (14.5,10) node[][scale=0.7]{$\OO_2,\OO_5$};
        \draw (23.5,9) node[][scale=0.7]{$\OO_1$};
         
         \path [draw,->>][] (15,1) -- (13.2,2);
        \draw (12.25,1.25) node[][scale=0.7]{$\OO_1$};

          \end{tikzpicture}
        \caption{Connection between non-base classes and their adjacent base classes belonging to $R_2$ via $C_4$-gadgets $\OO_2$ and $\OO_5$.\\
        The bi-directional arrow with an inscribed $G$ represents a connection using a $C_4$-gadget (introduced in $\OO_2$ and $\OO_5$).
        }
         \label{fig:lemma 23_b}
\end{subfigure}
\caption{ The directed  arrow (with double points \rotatebox[origin=c]{90}{$\rangle \rangle$}) from a set $A$ to $B$ represents that there is an oriented edge $(a,b)$ from any $a \in A$ to its any adjacent element $b\in B$.}
\label{fig:lemma 23}
\end{figure} 

\vspace{0.2cm}
  
We show an illustration of the introduced partial orientations in \Cref{fig:lemma 23}. The set $B$ is partitioned into three subsets as follows: (a) $R_1$: consisting of the elements of order $2$; (b) $R_2$: consisting of the elements of order $2^2$ and $p$, where $p$ is any odd prime in $\pi(G)$; (c) $R_3$: consisting of the elements of order $2^{\al}$, $\al\geq 3$, and $p^{\be}$, $\be \geq 2$, where $p$ is any odd prime in $\pi(G)$. Note that if $|G|$ is odd, then $G$ has no $2$-base element. Hence, the region $R_1$ does not exist, whereas $R_2$ and $R_3$ only contain $p$-base elements, where $p$ is an odd prime in $\pi(G)$.

\vspace{0.2cm}

    \textbf{Path directions:} 
    First, we list down some necessary observations, which can be argued similarly to the proof of \Cref{lem: pathtwo}. We also use \Cref{lem: pathtwo} for discussing path directions. In the following observations, $p$ is an odd prime in $\pi(G)$.

        \textit{Note 1:} 
        There is a directed path of length $2$ from any element of order $2^{\alpha}$ to any element of order $p$, using partial orientation $\OO_4$ (when $\alpha=1)$ or $\OO_5$ (when $\alpha=2)$ or $\OO_6$ (when $\alpha \geq 3)$ along with partial orientation $\OO_2$.

       \textit{Note 2:} 
       There is a directed path of length $2$ from any element of order $p^{\beta}, \ \beta \geq 2$ (or, of order $p$) to any element of order $2$ by noting \Cref{remark common element} and using $\OO_3$ (or $\OO_2$) together with $\OO_7$.
    
 Let $\Gamma=Pow(G)$ and $\OO$ denote the disjoint union of $\OO_1,\dots,\OO_7$ (or $\OO_1,\dots,\OO_3$ as required). Then,  we use the notation $\Gamma_{\OO}$ to denote the directed graph $(V(\Gamma),\OO)$. Moreover, let $d(a,b)$ (we use $d(a,b)$ instead of $d_{\Gamma_{\OO}}(a,b)$ as $\Gamma$ and $\OO$ are fixed in this context) denote the shortest distance from a vertex $a$ to a vertex $b$ in the directed graph $\Gamma_{\OO}$ and $d(a,S)=\min\{d(a,s) : s \in S\}$ denote the shortest distance from a vertex $a$ to a set $S$ in $\Gamma_{\OO}$. 

From \Cref{fig:lemma 23}, one can see that $d(v,e)=1$ for any non-base element $v$ and $d(e,u)=1$ for any base element $u$. This also implies that $d(v,u)\leq 2$, i.e., there is a directed path of length at most $2$ from any element $v\in NB$ to any element $u\in B$.

We claim that if $u \in B=R_1\cup R_2\cup R_3$, then $d(u,NB)=1$. For this, observe that if $u \in R_1$, then there exists some $v\in NB$ such that $(u,v) \in \OO_4$. Similarly, if $u \in R_2$, then there exists some $v \in NB$ such that $(u,v) \in \OO_2\cup \OO_5$ and if $u \in R_3$, then there exists some $v \in NB$ such that $(u,v) \in \OO_3 \cup \OO_6$. 

Noting $d(u,NB)=1$ for all $u\in B$ and $(v,e)\in \OO_1$ for all $v\in NB$, we have a directed path of length at most $2$ from any element of $B$ to $e$. Combining such a path with $(e,u') \in \OO_1$, where $u'$ is any element in $B$, we get a directed path of length at most $3$ between any two elements of $B$.

 Now, for any non-base element $v\in NB$, there exists at least one element $u \in R_2$ such that $(u,v)\in \OO_2$. Moreover, since $(e,u)\in \OO_1$ for all $u\in R_2 \subseteq B$, $d(e,v) = 2$ for all $v\in NB$. Now, as $(v',e)\in \OO_1$ for all $v'\in NB$, $d(v',v) \leq 1+d(e,v) = 3$, i.e., there is a directed path of length at most $3$ between any two elements in $NB$.


Now, we discuss the remaining case, i.e., when the source vertex $u$ is from $B$, and the destination vertex $v$ is from $NB$. Since $v$ is a non-base element, $o(v)$ always has at least one odd prime divisor, and hence there exists an element $a\in \langle v \rangle$ such that $o(a)$ is an odd prime. So the base class $[a]$ is in $R_2$ and participates in a $C_4$-gadget with the non-base class $[v]$ due to $\OO_2$ (see \Cref{fig:lemma 23}). Now, if $o(u)=2^{\alpha}$ ,$\alpha \geq 1$, then using Note 1, $d(u,a) \leq 2$ for all $a \in [a]$. Further using the $C_4$-gadget between $[a]$ and $[v]$, $d(u,v) \leq 3$. Now the case $o(u)=p^{\beta}$, $\beta \geq 1$ (where $p$ is an odd prime) is divided into two subcases according to the number of distinct odd prime divisors of $o(v)$. The first subcase is when $o(v)$ is divisible by at least two odd primes $p$ and $q$. Then, there exists an element $c \in \ang{v}$ of order $q$, and hence, there is a $C_4$-gadget between $[c]$ and $[v]$ due to $\OO_2$. Therefore, using the directed path of length $2$ from $u$ to any element of $[c]$ as described in \Cref{lem: pathtwo} and the gadget between $[c]$ and $[v]$, we have $d(u,v) \leq 3$. Now, consider the second subcase, i.e., when $o(u)=p^{\beta}$, $\beta \geq 1$ and $o(v)$ is divisible by only two primes $2$ and $p$. If $w$ is the (unique) element of order $2$ in $\ang{v}$, then Note 2 implies $d(u,w)\leq 2$. Moreover, since $w \in R_1$, $(w,v)$ belongs to $\OO_4$. This yields $d(u,v) \leq 3$ in this case.\hfill $\square$

\vspace{0.4cm}

\noindent \textbf{Proof of \Cref{nil: one 2-order}:} Due to \Cref{nil: od not 2} and \Cref{obs od is at most diam of partial orient}, it is sufficient to give a partial orientation of $Pow(G)$ with diameter $3$. The maximal cyclic subgroups of $G$ are of order $2^{\alpha}p^{\beta}$, where $1\leq \al \leq m$, $1\leq \be \leq n$ and $(\al,\be) \neq (m,n)$ (by \Cref{nilpotent}). Now, if $G$ has no maximal cyclic subgroup of order $2p^k$ for any $1 \leq k \leq n$, then we can use \Cref{nil: no mcs of order 2pk} to prove that $OD(Pow(G))=3$.

     Now we consider the case when $G$ has a maximal cyclic subgroup of order $2p^k$, for some $1\leq k \leq n$.
     Let $\ang{x}$ be the unique subgroup of order $2$ of $G$. Since every maximal cyclic subgroup of $G$ contains $\ang{x}$, each maximal cyclic subgroup of $G$ is of order $2p^{\beta}$, $1 \leq \beta \leq n$, by \Cref{obs mcs} (see \Cref{nn}). Note that, here $G=\ZZ_2\times S_p$, (where $S_p$ is the Sylow $p$-subgroup of $G$) due to Burnside's lemma (see \Cref{unique p-subgroup of a p-group}).

     Now we claim that in $G$, a base class $M_i$ of order $p^i$ is adjacent to exactly one non-base class $N_i$ of order $2p^i$ (see \Cref{adjacency of classes} for the definition of two GE-classes being adjacent). 
     One can verify this by using the facts that $G$ has a unique subgroup of order $2$, and the intersection of two cyclic subgroups is a cyclic subgroup. On the other hand, using \Cref{gtf_CLT}, one can verify that each $N_i$ is adjacent to exactly one $M_i$. So, there is a matching between the GE-classes of order $p^i$ and $2p^i$ for all $1 \leq i \leq n $ in $G$. Let $C_{ij} = M_{ij} \cup N_{ij}$, where $M_{ij}$ and $N_{ij}$ denote the $j$-th GE-class of order $p^i$ and $2p^i$ (since $G$ is not cyclic, $j>1$ for at least one $i$). Analogously, we match the elements $e$ and $x$ (recall that $x$ is the unique element of $G$ of order $2$) and put them in $C_0$. Now observe that $G$ can be viewed as a disjoint union of $C_0$ and the sets $C_{ij}$, where $1\leq i \leq n$ and $1< j$. We also partition each $N_{ij}$ in two non empty subsets $\{a_{ij}\}$, where $a_{ij}$ is an arbitrary element of $N_{ij}$ and $B_{ij}=N_{ij} \setminus \{a_{ij}\}$ (this can always be done since $|N_{ij}|\geq 2$). We now describe a partial orientation  $\OO$ in which we orient a subset of the edges in the subgraph induced by the set $C_0 \sqcup C_{ij}=C_0\sqcup(\{a_{ij}\}\sqcup B_{ij} \sqcup M_{ij})$, for each $i$ and each $j$ as follows:

         \vspace{0.2cm}
         
        \noindent $\OO$:  In this partial orientation, we put the following directed edges:
        
         (i) $(a_{ij},b)$, $(e,b)$ and $(b,x)$, for all $b \in B_{ij}$;
         
         (ii) $(a_{ij},e)$, $(e,x)$ and $(x,a_{ij})$;
         
         (iii) $(v,u)$, for all $v \in N_{ij}$ and for all $u \in M_{ij}$.
         
         See \Cref{fig: book like} for an illustration of $\OO$.

        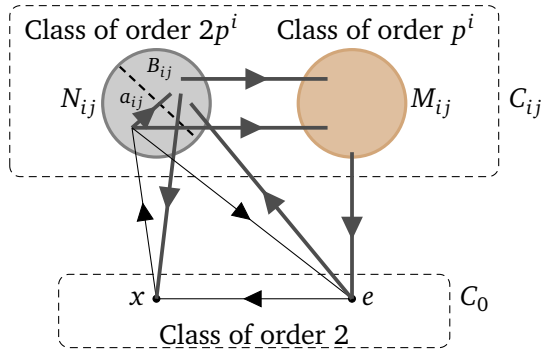
\begin{figure}[hpt!]
\centering

        \begin{tikzpicture}[scale=0.65]
        
        \coordinate (e) at (0,0);
        
        \coordinate (2) at (-4,0);
        
        \coordinate (p) at (-0.5,4.5);
        \coordinate (p') at (0,3);
        \coordinate (p'') at (-0.5,3.5);
        
        \coordinate (b) at (-3.5,4.5);
        \coordinate (b') at (-3.5,4.2);
        \coordinate (b'') at (-3.3,4);
        \coordinate (b''') at (-3.7,4.2);
        
        \coordinate (a) at (-4.5,3.5);

        \draw[color=black!100, densely dashed][rounded corners] (-7,2.5) rectangle (3,6);
        \draw[color=black!100, densely dashed][rounded corners](-6,-1) rectangle (2,0.5);
        \filldraw [black] (-4,0) circle(2pt);
        \filldraw [black] (0,0) circle(2pt);
        \filldraw[color=brown!70, fill=brown!50, very thick](0,4) circle [radius=1.1];
        \filldraw[color=black!50, fill=black!20, very thick](-4,4) circle [radius=1.1];

        \draw (e) node[right] {$e$};
        \draw (2) node[left] {$x$};
        \draw (-2,-0.3) node[below] {Class of order $2$};
        \draw (0.5,5) node[above] {Class of order $p^i$};
        \draw (1,4) node[right] {$M_{ij}$};
         \draw (-4.5,5) node[above] {Class of order $2p^i$};
        \draw (-5,4) node[left] {$N_{ij}$}; 
        \draw (-4.5,3.7) node[above][scale=0.75] {$a_{ij}$};
        \draw (-3.5,4.7) node[left][scale=0.75] {$B_{ij}$};
        \draw[line width=0.3 mm,densely dashed, black] (-3.25,3.35) -- (-4.75,4.75) ;

        \draw (e) -- (2) node[ currarrow, pos=0.5, xscale=-1, sloped, scale=1.5] {} ;
        \draw (e) -- (p') [line width=0.5mm, color=black!70] node[ currarrow, pos=0.5, xscale=-1, sloped, scale=1.5] {} ;
        \draw (a) -- (p'') [line width=0.5mm, color=black!70] node[ currarrow, pos=0.6, xscale=1, sloped, scale=1.5] {} ;
        \draw (b) -- (p) [line width=0.5mm, color=black!70] node[ currarrow, pos=0.5, xscale=1, sloped, scale=1.5] {} ;
        \draw (a) -- (b''') [line width=0.5mm, color=black!70] node[ currarrow, pos=0.35, xscale=1, sloped, scale=1.5] {} ;
        \draw (a) -- (2) node[ currarrow, pos=0.5, xscale=-1, sloped, scale=1.5] {} ;
        \draw (b') -- (2)[line width=0.5mm, color=black!70] node[ currarrow, pos=0.5, xscale=-1, sloped, scale=1.5] {} ;
        \draw (a) -- (e) node[ currarrow, pos=0.5, xscale=1, sloped, scale=1.5] {} ;
        \draw (b'') -- (e)[line width=0.5mm, color=black!70] node[ currarrow, pos=0.5, xscale=-1, sloped, scale=1.5] {} ;
        
        \draw (2,0) node[right] {$C_0$};
        \draw (3,4) node[right] {$C_{ij}$};
        \end{tikzpicture}
        \caption{Illustration of $\OO$.}
         \label{fig: book like}
\end{figure}

    \textbf{Path directions:}    
     From \Cref{fig: book like}, it can be observed that, using $\OO$, there is a directed path of length at most $2$ between any vertex of $C_0$ and any vertex of $C_{ij}$, for any $i$ and any $j$. 
    
    Note that any vertex $c \in C_{ij}$ has an outward edge either $(c,e)$ to $e$ or $(c,x)$ to $x$ (Recall that $x$ is the unique element of order $2$.). Now, we want to exhibit a directed path of length at most $3$ between two vertices $c \in C_{ij}$ and $c' \in C_{i'j'}$ where $i,i',j,j'$ are non-zero indices and $i$ (respectively $j$) may or may not be equal to $i'$ (respectively $j'$). Without loss of generality, let $(c, e) \in \OO$ (The other case can be argued similarly.). Then, we can use the edge $(c,e)$ together with the path from $e$ to $c'$ to have a path from $c$ to $c'$ of length at most $3$. Hence, it is shown that there is a directed path of length at most $3$ between any two vertices of $C \setminus C_0$. To have a directed path of length at most $3$ between any two vertices of $C_0$, we use the directed $3$-cycle $(e,x), (x,a_{ij}), (a_{ij},e)$ for any $i$, $j$.\hfill $\square$

~

{\bf Algorithm:} Given a nilpotent group $G$, it is easy to compute the oriented diameter of $Pow(G)$ with the help of the characterization given in this paper. We can compute the orders of each element in time linear in $|G|$ \cite{kavitha2007linear}. Once that is done, checking if a group is cyclic is easy. Checking if $G$ has multiple subgroups of prime order $p$ boils down to checking if it has at least $p$ elements of order $p$. A cyclic group $\ang{x}$ is maximal if it is not properly contained in $\ang{y}$ for any $y$. This can be tested in polynomial time. We note that nilpotency can be tested in polynomial time \cite{seress1997introduction}. \hfill $\lhd$

\section{Oriented Diameter of Enhanced Power Graphs and Commuting Graphs}\label{epow and com}

As a consequence of our results so far, one can easily note the following results regarding the oriented diameter of two widely studied and related graph classes, namely enhanced power graphs and commuting graphs. First, we provide the definitions of these two graphs.

\begin{definition}
\label{epow}
The enhanced power graph of a group $G$, denoted by $EPow(G)$, is an undirected graph with vertex set $G$, in which two vertices $x$ and $y$ are adjacent if and only if they are in a common cyclic subgroup of $G$, i.e., there exists $z$ in $G$ such that $x,y \in \langle z \rangle$. 
\end{definition}

\begin{definition}
\label{commuting}
    The commuting graph of a group $G$, denoted by $Com(G)$, is an undirected graph with vertex set $G$, in which $\{x,y\}$ is an edge if $xy=yx$ under the group operation.
\end{definition}

From definitions, one can easily note that $E(Pow(G)) \subseteq E(EPow(G)) \subseteq E(Com(G))$. Hence for a finite group, $OD(Com(G)) \leq OD(EPow(G)) \leq OD(Pow(G))$. Therefore, if $OD(Pow(G))\leq d$, then $d$ is an immediate upper bound for the oriented diameter of $Com(G)$ and as well as $EPow(G)$. Hence, from \Cref{4 upper bound}, we have the following straightforward corollary.

\begin{corollary}
    Let $G$ be a finite group without any maximal cyclic subgroup of order 2. Then, the oriented diameter of $EPow(G)$ and $Com(G)$ is at most $4$.
\end{corollary}

Moreover, $EPow(G)$ is a complete graph if and only if $G$ is cyclic, and $Com(G)$ is a complete graph if and only if $G$ is abelian. Hence, it makes sense to study the oriented diameter of enhanced power graphs of non-cyclic finite groups and commuting graphs of non-abelian finite groups. Now, from our previous discussion, it is clear that \Cref{od of all p-groups} and \Cref{nil: main result} yield upper bounds for the oriented diameter of corresponding enhanced power graphs and commuting graphs. But since there are more edges in $EPow(G)$ and $Com(G)$ than $Pow(G)$, there is a possibility that the actual value of the oriented diameter is less than these upper bounds. Hence, this leads to the following two natural questions.

\textbf{Question 1:} Can we characterize the oriented diameter of enhanced power graphs of non-cyclic finite nilpotent groups?

\textbf{Question 2:} Can we characterize the oriented diameter of commuting graphs of non-abelian finite nilpotent groups?

\bibliography{references}

\begin{thebibliography}{FMPR04}

\bibitem[AKC13]{abawajy2013power}
J.~Abawajy, A.~Kelarev, and M.~Chowdhury.
\newblock Power graphs: a survey.
\newblock {\em Electronic Journal of Graph Theory and Applications (EJGTA)}, 1(2):125--147, 2013.

\bibitem[BBRV21]{babu2021improvement}
J.~Babu, D.~Benson, D.~Rajendraprasad, and S.~N. Vaka.
\newblock An improvement to {C}hv{\'a}tal and {T}homassen’s upper bound for oriented diameter.
\newblock {\em Discrete Applied Mathematics}, 304:432--440, 2021.

\bibitem[Ber22]{bera2022line}
S.~Bera.
\newblock Line graph characterization of power graphs of finite nilpotent groups.
\newblock {\em Communications in Algebra}, 50(11):4652--4668, 2022.

\bibitem[Bur11]{burnside1911theory}
W.~Burnside.
\newblock {\em Theory of groups of finite order}.
\newblock University, 1911.

\bibitem[Cam10]{cameron2010power}
P.~J. Cameron.
\newblock The power graph of a finite group, ii.
\newblock {\em Journal of Group Theory}, 13:779--783, 2010.

\bibitem[Cam22]{cameron2022graphs}
P.~J. Cameron.
\newblock Graphs defined on groups.
\newblock {\em International Journal of Group Theory}, 11(2):53--107, 2022.

\bibitem[CCDS21]{cochran2021size}
G.~Cochran, {\'E}.~Czabarka, P.~Dankelmann, and L.~A. Sz{\'e}kely.
\newblock A size condition for diameter two orientable graphs.
\newblock {\em Graphs and Combinatorics}, 37:527--544, 2021.

\bibitem[CDS19]{czabarka2019degree}
{\'E}.~Czabarka, P.~Dankelmann, and L.~A. Sz{\'e}kely.
\newblock A degree condition for diameter two orientability of graphs.
\newblock {\em Discrete Mathematics}, 342(4):1063--1065, 2019.

\bibitem[CGS09]{chakrabarty2009undirected}
I.~Chakrabarty, S.~Ghosh, and M.~K. Sen.
\newblock Undirected power graphs of semigroups.
\newblock In {\em Semigroup Forum}, volume~78, pages 410--426. Springer, 2009.

\bibitem[Con14]{conrad2014generalized}
K.~Conrad.
\newblock Generalized quaternions.
\newblock {\em Retrieved form: https://kconrad. math. uconn. edu/blurbs/grouptheory/genquat. pdf}, 2014.

\bibitem[CT78]{chvatal1978distances}
V.~Chv{\'a}tal and C.~Thomassen.
\newblock Distances in orientations of graphs.
\newblock {\em Journal of Combinatorial Theory, Series B}, 24(1):61--75, 1978.

\bibitem[DF04]{dummit2004abstract}
D.~S. Dummit and R.~M. Foote.
\newblock {\em Abstract algebra}, volume~3.
\newblock Wiley Hoboken, 2004.

\bibitem[DG15]{doostabadi}
A.~Doostabadi and M.~Farrokhi~D. Ghouchan.
\newblock On the connectivity of proper power graphs of finite groups.
\newblock {\em Communications in Algebra}, 43(10):4305--4319, 2015.

\bibitem[EN09]{eggemann2009minimizing}
N.~Eggemann and S.~D. Noble.
\newblock Minimizing the oriented diameter of a planar graph.
\newblock {\em Electronic Notes in Discrete Mathematics}, 34:267--271, 2009.

\bibitem[FMPR04]{fomin2004free}
F.~V. Fomin, M.~Matamala, E.~Prisner, and I.~Rapaport.
\newblock At-free graphs: linear bounds for the oriented diameter.
\newblock {\em Discrete applied mathematics}, 141(1-3):135--148, 2004.

\bibitem[FMR04]{fomin2004complexity}
F.~V. Fomin, M.~Matamala, and I.~Rapaport.
\newblock Complexity of approximating the oriented diameter of chordal graphs.
\newblock {\em Journal of Graph Theory}, 45(4):255--269, 2004.

\bibitem[Fro95]{frobenius1895verallgemeinerung}
G.~Frobenius.
\newblock {\em Verallgemeinerung des Sylow'schen Satzes: {\"U}ber aufl{\"o}sbare Gruppen II. Von G. Frobenius}.
\newblock Reichsdr., 1895.

\bibitem[Gor80]{gorenstein1980finite}
D.~Gorenstein.
\newblock {\em Finite Groups}.
\newblock AMS/Chelsea Publication Series. Chelsea Publishing Company, 1980.

\bibitem[Hal18]{hall2018theory}
M.~Hall.
\newblock {\em The theory of groups}.
\newblock Courier Dover Publications, 2018.

\bibitem[Har79]{hardy1979introduction}
G.~H. Hardy.
\newblock {\em An introduction to the theory of numbers}.
\newblock Oxford Science Publication, 1979.

\bibitem[Kav07]{kavitha2007linear}
T.~Kavitha.
\newblock Linear time algorithms for abelian group isomorphism and related problems.
\newblock {\em Journal of Computer and System Sciences}, 73(6):986--996, 2007.

\bibitem[KLW10]{kwok2010oriented}
P.~K. Kwok, Q.~Liu, and D.~B. West.
\newblock Oriented diameter of graphs with diameter 3.
\newblock {\em J. Comb. Theory, Ser. B}, 100(3):265--274, 2010.

\bibitem[KSCC21]{kumar2021recent}
A.~Kumar, L.~Selvaganesh, P.~J. Cameron, and T.~T. Chelvam.
\newblock Recent developments on the power graph of finite groups--a survey.
\newblock {\em AKCE International Journal of Graphs and Combinatorics}, 18(2):65--94, 2021.

\bibitem[MRS14]{MR3200118}
A.~R. Moghaddamfar, S.~Rahbariyan, and W.~J. Shi.
\newblock Certain properties of the power graph associated with a finite group.
\newblock {\em J. Algebra Appl.}, 13(7):1450040, 18, 2014.

\bibitem[Pol14]{polymath2014variants}
D.~H.~J. Polymath.
\newblock Variants of the selberg sieve, and bounded intervals containing many primes.
\newblock {\em Research in the Mathematical sciences}, 1:1--83, 2014.

\bibitem[PPS21]{panda2021minimum}
R.~P. Panda, K.~L. Patra, and B.~K. Sahoo.
\newblock On the minimum degree of the power graph of a finite cyclic group.
\newblock {\em Journal of Algebra and Its Applications}, 20(03):2150044, 2021.

\bibitem[PPS23]{panda2023minimum}
R.~P. Panda, K.~L. Patra, and B.~K. Sahoo.
\newblock On the minimum degree of power graphs of finite nilpotent groups.
\newblock {\em Communications in Algebra}, 51(1):314--329, 2023.

\bibitem[Rob39]{robbins1939theorem}
H.~E. Robbins.
\newblock A theorem on graphs, with an application to a problem of traffic control.
\newblock {\em The American Mathematical Monthly}, 46(5):281--283, 1939.

\bibitem[Ros11]{rosen2011elementary}
K.~H. Rosen.
\newblock {\em Elementary number theory}.
\newblock Pearson Education London, 2011.

\bibitem[Rot12]{rotman2012introduction}
J.~J. Rotman.
\newblock {\em An introduction to the theory of groups}, volume 148.
\newblock Springer Science \& Business Media, 2012.

\bibitem[Ser97]{seress1997introduction}
A.~Seress.
\newblock An introduction to computational group theory.
\newblock {\em Notices of the AMS}, 44(6):671--679, 1997.

\bibitem[WC22]{WANG2022374}
X.~Wang and Y.~Chen.
\newblock Optimal oriented diameter of graphs with diameter 3.
\newblock {\em Journal of Combinatorial Theory, Series B}, 155:374--388, 2022.

\bibitem[Wes00]{west2001introduction}
D.~B. West.
\newblock {\em Introduction to Graph Theory}.
\newblock Prentice Hall, September 2000.

\bibitem[Zha14]{zhang2014bounded}
Y.~Zhang.
\newblock Bounded gaps between primes.
\newblock {\em Annals of Mathematics}, pages 1121--1174, 2014.

\end{thebibliography}
\bibliographystyle{alpha}


\appendix
\section{Extended Preliminary}\label{ext preli}
Let $X=(V,E)$ be a graph. If $S \subseteq V(X)$, then the subgraph with the vertex set $S$, and edges in $E(X)$ with both endpoints in $S$, is called the \textit{induced subgraph} of\ $X$ on $S$, and it is denoted by $X[S]$. In an undirected graph $X$, a vertex $u$ is said to be a \textit{neighbour} of a vertex $v$ (and vice versa) if $\{u,v\} \in E(X)$. 

In an undirected graph $X$ a \textit{path} between $u_1$ and $u_k$ is a sequence $u_1u_2\dots u_k$ of distinct vertices from $V(X)$ such that $\{u_{i},u_{i+1}\}\in E(X)$ for each $1\leq i \leq (k-1)$. The \textit{length} of a path is the number of edges participating in it, i.e., the length of the path $u_1u_2\dots u_k$ is $(k-1)$. A \textit{directed path} in a directed graph $\XX$ is defined analogously with the condition $(u_{i},u_{i+1})\in E(\XX)$ for each $1\leq i \leq (k-1)$.

We now state some useful group-theoretic facts. 

\begin{fact}
\label{gtf_commute}
    Let $x$ and $y$ be two non-trivial elements of a group $G$ such that $xy=yx$. Then $(xy)^n=x^ny^n$.
\end{fact}

\begin{fact}\label{co-prime orders commute}
    Let $G$ be a finite group and $x,y \in G \setminus \{e\}$ be two elements such that $o(x)$ and $o(y)$ are co-prime to each other and $xy=yx$. Then, $\langle xy \rangle$ forms a cyclic subgroup of $G$ of order $o(x)\cdot o(y)$. In particular, $o(xy)=o(x)\cdot o(y)$.
\end{fact}

The next property about finite nilpotent groups can be proved from the definition of finite nilpotent groups (see \Cref{preli}). 

\begin{fact}\label{nilpotent commute}\cite{dummit2004abstract}
    A finite group is nilpotent if and only if two elements with relatively prime orders commute with each other.
\end{fact}

\section{Appendix}

\subsection{Power graphs violating the conditions given by Czabarka et al. and Cochran et al.}
\label{app_degree_size_condition}

Let $n=p_1p_2p_3p_4\ldots p_r$ be a squarefree number where $p_1<p_2<\ldots<p_r$ are prime numbers and  $p_1=2, ~ p_2=3$ and $p_3=5$.
The degree of an element $x$ of order $2.3\dots p_{r-1}$ in $Pow(\ZZ_n)$ is strictly less than $n/2$ (See \cite[Sec. 1.2]{ panda2021minimum} for the degree expression of an element). Thus, this violates the condition given by Czabarka et al. \cite{czabarka2019degree}. Now, we show that this also violates the condition given by Cochran et al. \cite{cochran2021size}. 
 Note that each of the $\phi(o(x))=(p_1-1).(p_2-2).\ldots (p_{r-1}-1)$ generators of $\langle x \rangle $ has degree less than $n/2$.
Hence, considering the missing edges due to the generators of $\langle x \rangle$, we can show that $(\binom{n}{2}-n+5)-|E(Pow(\ZZ_n))| \geq n. \{\frac{ (p_1-1).(p_2-1).\dots(p_{r-1}-1)}{2}-1\}+5$. Now the term $\frac{ (p_1-1).(p_2-1).\dots(p_{r-1}-1)}{2}$ is at least  $4$ for $r\geq 4$. Therefore, the power graphs of such cyclic groups contain at least $3n$ edges fewer than the threshold size of the edge set provided by Cochran et al. \cite{cochran2021size}. A similar situation arises for $n=3.5.7.11.13.k'$, where $k'$ is a squarefree natural number such that $gcd(3.5.7.11.13,k')=1$.

Moreover, the difference depends on the term $T=\phi(n)/ (p_r-1)$. One can show that $T$ can be huge for sufficiently large $n$. For that, we consider $H_1=\liminf_{\infty} (p_{k+1}-p_k)$, where $p_i$ denotes the $i$-th prime. It is known that $H_1 \leq 246$ \cite{zhang2014bounded, polymath2014variants}. Hence, for infinitely many choices of $n$, the difference between its largest two prime factors $p_r - p_{r-1} \leq (n-1)p_r$, which in turn gives $p_r \leq n^{1/2}$. On the other hand, for every $c >0$ and for sufficiently large $n$, $\phi(n) > c. n ^{1-\delta}$, where $\delta > 0$ (see \cite [Th. 327]{hardy1979introduction}). Taking $\delta=0.001$, we get $T > c.n^{0.499}$.

\subsection{Power graph of $\ZZ_6$}\label{app_cyclic_Z6}
\begin{lemma}\label{cyclic_Z6}
     The oriented diameter of $Pow(\mathbb{Z}_6)$ is $3$.
\end{lemma}
\begin{figure}[hpt!]

\centering
\begin{subfigure}[b]{0.45\textwidth}
\centering
\begin{tikzpicture}[scale=0.35]

                
\coordinate (0) at (0,0);
\coordinate (3) at (0,6);
\coordinate (1) at (2,2);
\coordinate (2) at (2,4);
\coordinate (5) at (-2,2);
\coordinate (4) at (-2,4);

\filldraw [black] (0) circle(2pt);
\filldraw [black] (1) circle(2pt);
\filldraw [black] (2) circle(2pt);
\filldraw [black] (3) circle(2pt);
\filldraw [black] (4) circle(2pt);
\filldraw [black] (5) circle(2pt);

\draw (0) node[right][scale=1] {\texttt{0}};
\draw (1) node[right][scale=1] {\texttt{1}};
\draw (2) node[right][scale=1] {\texttt{2}};
\draw (3) node[right][scale=1] {\texttt{3}};
\draw (4) node[left][scale=1] {\texttt{4}};
\draw (5) node[left][scale=1] {\texttt{5}};

\draw (0) -- (1) ;
\draw (0) -- (2) ;
\draw (0) -- (3) ;
\draw (0) -- (4) ;
\draw (0) -- (5) ;
\draw (1) -- (2) ;
\draw (1) -- (3) ;
\draw (1) -- (4) ;
\draw (1) -- (5) ;
\draw (2) -- (4) ;
\draw (2) -- (5) ;
\draw (4) -- (5) ;
\draw (5) -- (3) ;



                \end{tikzpicture}

\caption{Power graph of $\mathbb{Z}_6$.}
\label{Fig_Z6_a}
\end{subfigure}
\begin{subfigure}[b]{0.45\textwidth}
\centering
\begin{tikzpicture}[scale=0.35]

                
\coordinate (0) at (0,0);
\coordinate (3) at (0,6);
\coordinate (1) at (2,2);
\coordinate (2) at (2,4);
\coordinate (5) at (-2,2);
\coordinate (4) at (-2,4);

\filldraw [black] (0) circle(2pt);
\filldraw [black] (1) circle(2pt);
\filldraw [black] (2) circle(2pt);
\filldraw [black] (3) circle(2pt);
\filldraw [black] (4) circle(2pt);
\filldraw [black] (5) circle(2pt);

\draw (0) node[right][scale=1] {\texttt{0}};
\draw (1) node[right][scale=1] {\texttt{1}};
\draw (2) node[right][scale=1] {\texttt{2}};
\draw (3) node[right][scale=1] {\texttt{3}};
\draw (4) node[left][scale=1] {\texttt{4}};
\draw (5) node[left][scale=1] {\texttt{5}};
\path [draw,->][blue] (-0.75,5) -- (-1.1,4.25);
\path [draw,->][blue] (-0.25,4.25) -- (-0.25,5);
\path [draw,->][blue] (0.65,4.25) -- (0.25,5);
\path [draw,->][blue] (-2.25,2.5) -- (-2.25,3.25);
\path [draw,->][blue] (0.1,3.25) -- (0.85,3.7);
\path [draw,->][blue] (0.1,2.20) -- (0.85,2.20);
\path [draw,->][blue] (-2,1.60) -- (-1.3,1);

\draw (0) -- (1) ;
\draw (0) -- (2) ;
\draw (0) -- (3) ;
\draw (0) -- (4) ;
\draw (0) -- (5) ;
\draw (1) -- (2) ;
\draw (1) -- (3) ;
\draw (1) -- (4) ;
\draw (1) -- (5) ;
\draw (2) -- (4) ;
\draw (2) -- (5) ;
\draw (4) -- (5) ;
\draw (5) -- (3) ;



                \end{tikzpicture}

\caption{The directions discussed in the proof.}
\label{Fig_Z6_b}
\end{subfigure}

\caption{Illustration of proof of \Cref{cyclic_Z6}}
\label{Fig_Z6}
\end{figure}
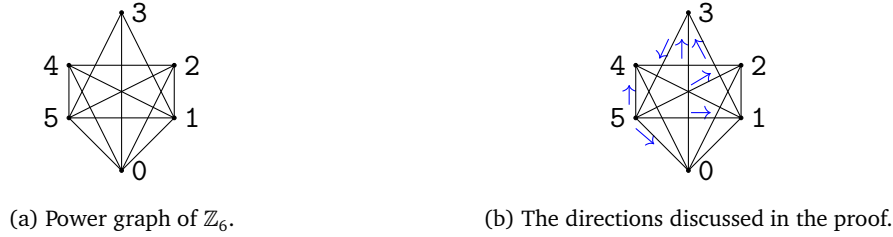
\begin{proof} 
      Due to  \Cref{Upper3}, it is sufficient to prove that $Pow(\mathbb{Z}_6)$ (\Cref{Fig_Z6_a}) cannot have oriented diameter $2$. The cyclic group $\ZZ_6$ has group elements \texttt{0, 1, 2, 3, 4, 5}. In any strong orientation, all the edges incident to a vertex $v$ cannot be directed outwards from $v$ or inwards to $v$. Now in the case when the vertex \texttt{3} in $Pow(\mathbb{Z}_6)$ has only one outward edge, we can assume without loss of generality that the directions given to the edges incident to \texttt{3} in the graph $Pow(\ZZ_6)$ are \texttt{(3,5), (0,3)} and \texttt{(1,3)}. Then, to have a directed path of length $2$ from \texttt{3} to the vertices  \texttt{ 4, 2, 1, 0}, we need the following directed edges: \texttt{(5,4),(5,2),(5,1),(5,0)} respectively. In that case, we can not have a directed path of length $2$ from \texttt{2} to \texttt{5} (see \Cref{Fig_Z6_b}), and hence, $Pow(\ZZ_6)$ cannot have an orientation with diameter $3$. The case when vertex \texttt{3} in $Pow(\mathbb{Z}_6)$ has only one inward edge is similar. 
\end{proof}

\subsection{Proof of \Cref{direct product of co-prime order}}\label{app_direct product of co-prime order}

\vspace{0.2cm}

    Since $g_1$ generates $g_2$ in $G$, there exists natural number $k_1 \geq 1$ such that $g_1^{k_1}=g_2$. Moreover, $g_1^{k_1+m_1\cdot o(g_1)}=g_2$, where $m_1$ is an integer. Since $h_1$ generates $h_2$ in $H$, we can similarly write that $h_1^{k_2+m_2\cdot o(h_1)}=h_2$, where $k_2\geq 1$ is a natural number and $m_2$ is an integer. The element $(g_1,h_1)$ generates $(g_2,h_2)$ if and only if there is an integer $x$ such that $(g_1,h_1)^x=(g_2,h_2)$, i.e., $g_1^X=g_2$ and $h_1^x=h_2$. Such $x$ exists if the congruence equations $x\equiv k_1$ (mod $o(g_1)$) and $x\equiv k_2$ (mod $o(h_1)$) have a solution. Now, $gcd(|G|,|H|)=1$ implies that $o(g_1)$ and $o(h_1)$ are co-prime to each other. So, by the Chinese Remainder Theorem (see \cite{rosen2011elementary}), the above equations have a solution, say $l$, and we can write $(g_1,h_1)^{l}=(g_2,h_2)$. Hence, $(g_1,h_1)$ generates $(g_2,h_2)$ in $G\times H$.

\subsection{An observation for the proof of \Cref{nil: OD strictly greater than 3} and \Cref{nil: one 2-order}}\label{nn}

\vspace{0.2cm}

\begin{observation}\label{obs mcs}
    Let $G$ be a non-cyclic nilpotent group and $|G|=2^mp^n$, where $p$ is an odd prime and $m,n \geq 1$. If $C$ is a maximal cyclic subgroup of $G$ of order $2p^{\beta}$, $1\leq \be \leq n$, containing a base element $x$ of order $2$, then any maximal cyclic subgroup of $G$ containing $x$ is of order $2p^{\gamma}$, for some $\gamma \geq 1$.
\end{observation}

\begin{proof}
    Let $y \in C$ be an element of order $p^{\beta}$. Note that $xy$ generates $C$. For the sake of contradiction, we assume that a maximal cyclic subgroup $C'$ of $G$ containing $x$ has order $2^{\al}p^{\gamma}$ where $\al >1, \gamma \geq 1$. Hence, $C'$ must have an element $w$ of order $2^2$ (by \Cref{gtf_CLT}), and $w$ must generate $x$ (because $w$ generates an element of order $2$ of $C'$, and $x$ is the only element of order $2$ in $C'$). Now, using \Cref{nilpotent commute}, $w$ and $y$ commute with each other. Hence, using \Cref{gtf_commute} and a well-known number theoretic fact\footnote{The congruence equation $az\equiv b \ (mod \ n)$ has a solution for $z$ if and only if $gcd(a,n)$ divides $b$.}, $wy$ generates both $y$ and $w$. This implies that $wy$ generates both $x$ (since $w$ generates $x$) and $y$. Therefore, $wy$ generates $xy$. So, $wy$ generates $C$. This contradicts that $C$ is a maximal cyclic subgroup of $G$.
    \end{proof}

\subsection{Examples of nilpotent groups corresponding to \Cref{nil: main result}}
\label{nil: example}

\vspace{0.2cm}

 For each of the conditions of \Cref{nil: main result}, we provide examples of finite non-cyclic nilpotent groups that are not in $\Gprime$.
\begin{itemize}
    \item  Only condition (a): $G=G_1 \times G_2 $ where $G_1$ is a $p$-group and $G_2$ is a $q$-group, where $p$ and $q$ are odd primes.
    \item  Only condition (b): $G= \mathbb{Z}_{4p^n} \times \mathbb{Z}_{4q^m} $ where $p$ and $q$ are odd primes, and $m,n \geq 1$.
    \item Only condition (c): $G= G_1 \times \mathbb{Z}_{p^n} $ where $G_1$ is a $2$-group, $p$ is an odd prime and $n \geq 1$.
    \item  Only condition (d): $G= Q_8 \times G_1 $ where $G_1$ is a $p$-group, and $p$ is an odd prime.
    \item  None of (a)-(d): $G= \mathbb{Z}_{2p} \times \mathbb{Z}_{2p} $ where $p$ is an odd prime.
\end{itemize}

\subsection{Proof of \Cref{nil: no mcs of order 2pk} }
\label{app_nil_no mcs of order 2pk}

\vspace{0.2cm}

 Due to \Cref{nil: od not 2} and \Cref{obs od is at most diam of partial orient}, it is sufficient to give an orientation of $Pow(G)$ with diameter $3$. Along with the partial orientations $\OO_1,\OO_2,\OO_3$ given in \Cref{Remark_2}, we also use the following partial orientations:
    \begin{itemize}
        
        \item[$\OO_4$ :] From any base element of order $2$, we orient the edges towards all the adjacent non-base elements of order $2p^{\beta}$ (where $\beta \geq 1$).
        
        \item [$\OO_5$ :] Consider a non-base class $N$ of order $2^\al p^\be$, $\alpha \geq 2, \beta \geq 1$. By \Cref{gtf_CLT}, $N$ is adjacent to only one base class $M$ of order $2^2$. Then, we orient the edges of $E(M,N)$ as described in $\OO_2$ of \Cref{Remark_2}. The choices of anchor gadget points of $N$ for $M$ depend on $M$ as stated in \Cref{Remark_2}. In other words, while introducing $\OO_5$ in a non-base class $N$ of order $2^{\al}p^{\be}$, $\al \geq 2, \be \geq 1$, we select a pair of gadget anchor points $\{n_1,n_2\}$ in $N$ for the base class of order $2^2$ adjacent to $N$ in such a way that neither $n_1$ nor $n_2$ has used as gadget anchor point in $N$ while introducing $\OO_2$ for the base class of order $p$ adjacent to $N$. This is possible since only one base class of odd prime order is adjacent to $N$ and also $|N|=\phi(2^{\al}p^{\be})>2^2 = 2\cdot2$.

        \item[$\OO_6$ :] From any base element of order $2^{\al}, \al \geq 3$, we orient the edges towards all the adjacent non-base elements. Note that these non-base elements are of order $2^{\delta}p^{\beta}$, where $\delta \geq \al$ and $\beta \geq 1$.    

        \item[$\OO_7$ :] From any non-base element of order $2^2p^{\beta}$, $\be \geq 1$, we orient the edge towards the (unique) adjacent element of order $2$. 
    \end{itemize}
 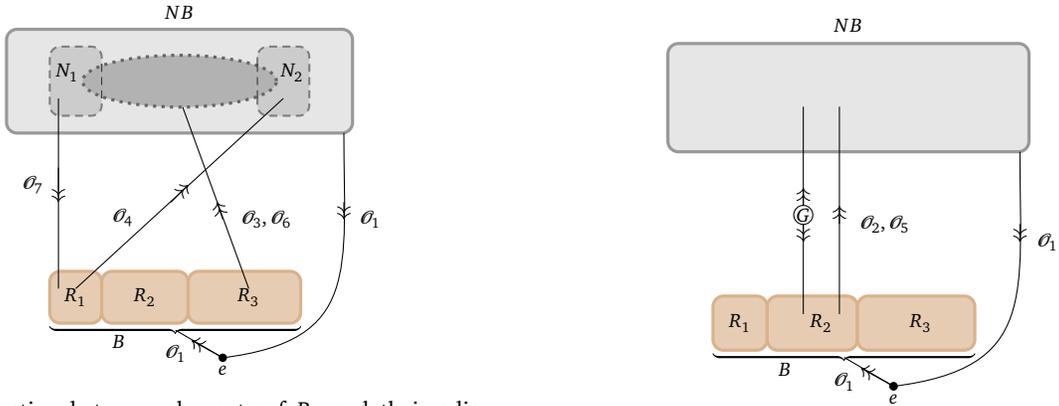
\begin{figure}[hpt!]
\begin{subfigure}[b]{0.47\textwidth}
\centering
        \begin{tikzpicture}[scale=0.23]
        \coordinate (a) at (2,7);
        \coordinate (b) at (3,11);
        \coordinate (c) at (4,7);
        \coordinate (d) at (4.1,11);
        \coordinate (e) at (10,7);
        \coordinate (f) at (10,11);
        \coordinate (g) at (12,7);
        \coordinate (h) at (11.1,11);

        \filldraw[color=black!40, fill=black!10, very thick][rounded corners] (2.5,20) rectangle (22.5,14);

         \filldraw[color=brown!60, fill=brown!40, very thick][rounded corners] (5,6) rectangle (8,3);
        \filldraw[color=brown!60, fill=brown!40, very thick][rounded corners] (8,6) rectangle (13,3);
        \filldraw[color=brown!60, fill=brown!40, very thick][rounded corners] (13,6) rectangle (19.5,3);

        \filldraw[color=black!50, fill=black!20, thick, densely dashed][rounded corners] (5,19) rectangle (8,15);
        \filldraw[color=black!50, fill=black!20, thick, densely dashed][rounded corners] (17,19) rectangle (20,15);

        \path [draw] (15,1) -- (12.26,2.6);
        \path [draw] (5.5,5) -- (5.5,16);
        \path [draw] (6.5,5) -- (18.5,16);
        \path [draw] (16.5,5) -- (12.5,16);

        \path [draw] (15,1) .. controls (23,2) and (22,6) ..  (22,14);
        \path [draw,<<-][] (22,9) -- (22,10);
        
         \path [draw,->>][] (5.5,12) -- (5.5,10);
        \path [draw,->>][] (12,10) -- (13,11);
        \path [draw,->>][] (15,9) -- (14.67,10);

   
                   \filldraw [black] (15,1) circle(6pt);

         \filldraw[color=black!60, fill=black!30, very thick, dotted] (12.5,17) ellipse (5.6cm and 1.5cm);
         
          \draw [color=black!80, densely dashed] (8,18) -- (8,16); 
          \draw [color=black!80, densely dashed] (17,18) -- (17,16);

  \draw [decorate,
    decoration = {calligraphic brace}][thick] (19.5,2.8) --  (5,2.8);

        \draw (12.5,21) node[][scale=0.7] {$NB$};
        \draw (9,2.7) node[below][scale=0.7]{$B$};
        \draw (6.5,4.5) node[][scale=0.7]{$R_1$};
        \draw (10.5,4.5) node[][scale=0.7]{$R_2$};
        \draw (16.5,4.5) node[][scale=0.7]{$R_3$};
        \draw (6,17.5) node[][scale=0.7]{$N_1$};
        \draw (19,17.5) node[][scale=0.7]{$N_2$};
        \draw (15,0.3) node[][scale=0.7]{$e$};
        \draw (4,11) node[][scale=0.7]{$\OO_7$};
        \draw (9.2,9.1) node[][scale=0.7]{$\OO_4$};
        \draw (17.5,9) node[][scale=0.7]{$\OO_3, \OO_6$};
        \draw (23.5,9) node[][scale=0.7]{$\OO_1$};
        
        \draw (12.25,1.25) node[][scale=0.7]{$\OO_1$};
         \path [draw,->>][] (15,1) -- (13.2,2);
              
         \end{tikzpicture}
        \caption{Connection between elements of $R_1$ and their adjacent elements in $N_1$ and $N_2$, where  $N_1=\{v \in NB | o(v)=2^2p^{\be},\ \be \geq 1 \}$ and $N_2=\{v \in NB | o(v)=2p^{\be},\ \be \geq 1\}$.\\ Connection between elements of $R_3$ and their adjacent elements in $NB$. }
         \label{fig:lemma 23_a}
\end{subfigure} \hspace{1 cm}
\begin{subfigure}[b]{0.47\textwidth}
\centering
 \begin{tikzpicture}[scale=0.24]

        \filldraw[color=black!40, fill=black!10, very thick][rounded corners] (2.5,20) rectangle (22.5,14);

        \filldraw[color=brown!60, fill=brown!40, very thick][rounded corners] (5,6) rectangle (8,3);
        \filldraw[color=brown!60, fill=brown!40, very thick][rounded corners] (8,6) rectangle (13,3);
        \filldraw[color=brown!60, fill=brown!40, very thick][rounded corners] (13,6) rectangle (19.5,3);


        \path [draw] (15,1) -- (12.26,2.6);
%
%
%
%

\path [draw] (10,5) -- (10,16.5);
            \path [draw] (12,5) -- (12,16.5); 
        \path [draw,->>][] (12,10) -- (12,11);
                \path [draw] (15,1) .. controls (23,2) and (22,6) ..  (22,14);
        \path [draw,<<-][] (22,9) -- (22,10);
         \filldraw [white] (10,10.5) circle(15pt);
         \draw [] (10,10.5) circle(15pt);  
         \draw (10,10.5) node[][scale=0.7] {$G$};
         \path [draw,->>][] (10,11) -- (10,12);
        \path [draw,->>][] (10,10) -- (10,9);
                  \filldraw [black] (15,1) circle(6pt);

  \draw [decorate,
    decoration = {calligraphic brace}][thick] (19.5,2.8) --  (5,2.8);

        \draw (12.5,21) node[][scale=0.7] {$NB$};
        \draw (9,2.7) node[below][scale=0.7]{$B$};
        \draw (6.5,4.5) node[][scale=0.7]{$R_1$};
        \draw (11,4.5) node[][scale=0.7]{$R_2$};
        \draw (16.5,4.5) node[][scale=0.7]{$R_3$};
        \draw (15,0.3) node[][scale=0.7]{$e$};
        \draw (14.5,10) node[][scale=0.7]{$\OO_2,\OO_5$};
        \draw (23.5,9) node[][scale=0.7]{$\OO_1$};
         
         \path [draw,->>][] (15,1) -- (13.2,2);
        \draw (12.25,1.25) node[][scale=0.7]{$\OO_1$};

          \end{tikzpicture}
        \caption{Connection between non-base classes and their adjacent base classes belonging to $R_2$ via $\OO_2$ and $\OO_5$.\\
        The bi-directional arrow with an inscribed $G$ represents a connection using a $C_4$-gadget (introduced in $\OO_2$ and $\OO_5$).
        }
         \label{fig:lemma 23_b}
\end{subfigure}
\caption{ The directed  arrow (with double points \rotatebox[origin=c]{90}{$\rangle \rangle$}) from a set $A$ to $B$ represents that there is an oriented edge $(a,b)$ from any $a \in A$ to its any adjacent element $b\in B$.}
\label{fig:app_od3}
\end{figure}

     We show an illustration of the given orientations in \Cref{fig:app_od3}. The set $B$ is partitioned into three subsets as follows: (a) $R_1$: consisting of the elements of order $2$; (b) $R_2$: consisting of the elements of order $2^2$ and $p$; (c) $R_3$: consisting of the elements of order $2^{\al}$, $\al\geq 3$, and $p^{\be}$, $\be \geq 2$.
     
   \vspace{0.2cm}

     \textbf{Path directions:} First, we point out the following observations, which can be argued similarly to \Cref{lem: pathtwo}:
    
    \textit{Note 1:}  
    There is a directed path of length $2$ from any element of order $2^{\alpha}$ to any element of order $p$, using $\OO_4$ (when $\alpha=1)$ or $\OO_5$ (when $\alpha=2)$ or $\OO_6$ (when $\alpha \geq 3)$ along with using $\OO_2$.

    \textit{Note 2:} There is a directed path of length $2$ from any element of order $p^{\be}$, $\be \geq 1$ to any element of order $2^2$ using $\OO_2$ (when $\be=1$) or $\OO_3$ (when $\be \geq 2$) along with using $\OO_5$.

    \textit{Note 3:} There is a directed path of length $2$ from any element of order $p^{\be}$, $\be \geq 1$ to any element of order $2$ using $\OO_2$ (when $\be=1$) or $\OO_3$ (when $\be \geq 2$) along with using $\OO_7$.

Let $\Gamma=Pow(G)$ and $\OO$ denote the disjoint union of $\OO_1,\dots,\OO_6$. Then,  we use the notation $\Gamma_{\OO}$ to denote the directed graph $(V(\Gamma),\OO)$. Moreover, let $d(a,b)$ (we use $d(a,b)$ instead of $d_{\Gamma_{\OO}}(a,b)$ as $\Gamma$ and $\OO$ are fixed in this context) denote the shortest distance from a vertex $a$ to a vertex $b$ in the directed graph $\Gamma_{\OO}$ and $d(a,S)=\min\limits_{s \in S}d(a,s)$ denote the shortest distance from a vertex $a$ to a set $S$ in $\Gamma_{\OO}$.

Although other than the path direction from a base element to a non-base element, the path directions are the same as those discussed in the proof of \Cref{nil: two odd primes}, we discuss them here also for the sake of completeness.
From \Cref{fig:app_od3}, one can see that $d(v,e)=1$ for any non-base element $v$ and $d(e,u)=1$ for any base element $u$. This also implies that $d(v,u) \leq 2$, i.e., there is a directed path of length at most $2$ from any element $v\in NB$ to any element $u\in B$.

We claim that if $u \in B=R_1\cup R_2\cup R_3$ then $d(u,NB)=1$. For this, observe that if $u \in R_1$, then there exists some $v\in NB$ such that $(u,v) \in \OO_4$. Similarly, if $u \in R_2$, then there exists some $v \in NB$ such that $(u,v) \in \OO_2\cup \OO_5$ and if $u \in R_3$, then there exists some $v \in NB$ such that $(u,v) \in \OO_3 \cup \OO_6$. 

Noting $d(u,NB)=1$ for all $u\in B$ and $(v,e)\in \OO_1$ for all $v\in NB$, we have a directed path of length at most $2$ from any element of $B$ to $e$. Combining such a path with $(e,u')\in \OO_1$, where $u'$ is any element in $B$, we get a directed path of length at most $3$ between any two elements of $B$.

 Now, for any non-base element $v\in NB$, there exists at least one element $u \in R_2$ such that $(u,v)\in \OO_2$. Moreover since $(e,u)\in \OO_1$ for all $u\in R_2 \subseteq B$, we have $d(e,v)= 2$ for all $v\in NB$. Now, as $(v',e)\in \OO_1$ for all $v'\in NB$, we get $d(v',v)\leq 1+d(e,v) = 3$, i.e., there is a directed path of length at most $3$ between any two elements in $NB$.


Now, the only case that remains to be discussed is when the source vertex $u$ is from $B$, and the destination vertex $v$ is from $NB$. Since $v$ is a non-base element, $ p \mid o(v)$, and hence there exists an element $a\in \langle v \rangle$ such that $o(a)=p$. So the base class $[a]$ is in $R_2$ and participates in a $C_4$-gadget with the non-base class $[v]$ due to $\OO_2$ (see \Cref{fig:lemma 23}). Now, if $o(u)=2^{\alpha}$, $\alpha \geq 1$, then using Note 1, we have $d(u,a') \leq2$ for all $a' \in [a]$. Further using the $C_4$-gadget between $[a]$ and $[v]$, we have $d(u,v) \leq 3$. Else, consider the case when $o(u)=p^{\beta}$, $\beta \geq 1$. If $2^2 \nmid o(v)$, then Note 2 implies $d(u,b) \leq 2$, where $b$ is the (unique) element of order $2$ in $\ang{v}$. Moreover, since $b \in R_1$, we have $(b,v) \in \OO_4$. This gives us $d(u,v) \leq 3$ in this case. If $2^2 \mid o(v)$, then there exists an element $c\in \ang{v}$ of order $2^2$ and $[c]$ participates in a $C_4$-gadget with $[v]$. Now, using Note 3, we have $d(u,c') \leq 2$ for all $c' \in [c]$. After that, due to the $C_4$-gadget between $[c]$ and $[v]$ we have $d(u,v) \leq 3$.

\subsection{Proof of \Cref{nil: one p-order}}\label{app_nil_one p-order}


\vspace{0.2cm}

 Due to \Cref{nil: od not 2} and \Cref{obs od is at most diam of partial orient}, it is sufficient to give an orientation of $Pow(G)$ with diameter $3$. For that, along with the partial orientations $\OO_1,\OO_2,\OO_3$ discussed in \Cref{Remark_2}, we use the following partial orientation.

\begin{itemize}

        \item[$\OO_4$:] From any base element of order $2^{\al}, \al \geq 1$, we orient the edges towards all the adjacent non-base elements. Note that these non-base elements are of order $2^{\delta}p^{\beta}$, where $\delta \geq \al$ and $\beta \geq 1$.         
    \end{itemize}

    \vspace{0.1cm}
    
    \noindent \textbf{Path directions:} First, we point out the following observations, which can be argued similarly to \Cref{lem: pathtwo}:
    
   \textit{Note 1:} Using $\OO_4$ together with $\OO_2$, there is a directed path of length $2$ from any base element of order $2^{\alpha}$, $\alpha \geq 1$ to any base element of order $p$. 
    
    \textit{Note 2: } Using $\OO_3$ together with $\OO_2$, there is a directed path of length $2$ from any base element of order $p^{\be}$, $\be \geq 2$ to any base element of order $p$.
    
Let $\Gamma=Pow(G)$ and $\OO$ denote the disjoint union of $\OO_1,\dots,\OO_4$. Then, we use the notation $\Gamma_{\OO}$ according to \Cref{def: partial orientation}. Moreover, let $d(a,b)$ (we use $d(a,b)$ instead of $d_{\Gamma_{\OO}}(a,b)$ as $\Gamma$ and $\OO$ are fixed in this context) denote the shortest distance from a vertex $a$ to a vertex $b$ in the directed graph $\Gamma_{\OO}$ and $d(a,S)=\min\limits_{s \in S}d(a,s)$ denote the shortest distance from a vertex $a$ to a set $S$ in $\Gamma_{\OO}$.

Although other than the path direction from a base element to a non-base element, the path directions are the same as those discussed in the proof of \Cref{nil: two odd primes}, we discuss them here also for the sake of completeness.
One can see that $d(v,e)=1$ for any non-base element $v$ and $d(e,u)=1$ for any base element $u$. This also implies that $d(v,u) \leq 2$, i.e., there is a directed path of length at most $2$ from any element $v\in NB$ to any element $u\in B$.

We claim that if $u \in B$, then $d(u,NB)=1$. For this, observe that if $o(u)=2^{\al}$, $\al \geq 1$, then there exists some $v\in NB$ such that $(u,v) \in \OO_4$. Similarly, if $o(u)= p$, then there exists some $v \in NB$ such that $(u,v) \in \OO_2$ and if $o(u)=p^{\be}$, $\be \geq 2$, then there exists some $v \in NB$ such that $(u,v) \in \OO_3$. 

Noting $d(u,NB)=1$ for all $u\in B$ and $(v,e)\in \OO_1$ for all $v\in NB$, we have a directed path of length at most $2$ from any element of $B$ to $e$. Combining such a path with $(e,u')\in \OO_1$, where $u'$ is any element in $B$, we get a directed path of length at most $3$ between any two elements of $B$.

 Now, for any non-base element $v\in NB$, there exists at least one element $u \in B$ such that $o(u)=p$ and $(u,v)\in \OO_2$. Moreover since $(e,u)\in \OO_1$ for all $u\in B$, we have $d(e,v) = 2$ for all $v\in NB$. Now, as $(v',e)\in \OO_1$ for all $v'\in NB$, we get $d(v',v) \leq 1+d(e,v) = 3$, i.e., there is a directed path of length at most $3$ between any two elements in $NB$.

Now, the only case that remains to be discussed is when the source vertex $u$ is from $B$, and the destination vertex $v$ is from $NB$. At first, observe that since by assumption $G$ has a unique subgroup of order $p$, it has only one base class $[a]$ of order $p$. Also, the base class $[a]$ participates in a $C_4$-gadget with the non-base class $[v]$ due to $\OO_2$. Now, if $o(u)=2^{\alpha}$, $\alpha \geq 1$, then using Note 1, we have $d(u,a')=2$ for all $a' \in [a]$. Further using the $C_4$-gadget between $[a]$ and $[v]$, we have $d(u,v) \leq 3$. If $o(u)=p^{\beta}$, $\beta \geq 2$, then using Note 2 and the $C_4$-gadget between $[a]$ and $[v]$, we have $d(u,v) \leq 3$. If $o(u)=p$, it is easy to observe that $u \in [a]$. Now, we use the directed edges between $[a]$ and $[v]$, which are in $\OO_2$. This gives a directed path from any $u \in [a]$ to $v \in NB$ of length at most $3$ (see \Cref{fig:gmgn}).

\end{document}